\documentclass[table,dvipsnames]{amsart}

\usepackage[margin=1in]{geometry}
\usepackage{amsmath,amssymb,amsthm}
\usepackage{mathtools}
\usepackage{centernot}
\usepackage{tikz}
\usepackage{tikz-cd}
\usepackage{multirow}
\usepackage{tabu}
\usepackage{setspace}
\usepackage{subfig}
\usepackage{caption}

\newcommand{\F}{\mathbb{F}}
\newcommand{\Q}{\mathbb{Q}}
\newcommand{\Qbar}{\overline{\Q}}
\newcommand{\Aut}{\operatorname{Aut}}
\newcommand{\Gal}{\operatorname{Gal}}
\newcommand{\SL}{\operatorname{SL}}
\newcommand{\GL}{\operatorname{GL}}
\newcommand{\PGL}{\operatorname{PGL}}
\newcommand{\Z}{\mathbb{Z}}
\newcommand{\Kbar}{\overline K}
\newcommand{\tor}{\mathrm{tors}}
\newcommand{\I}{\mathcal{I}}
\newcommand{\angles}[1]{\langle #1\rangle}

\newcommand{\tLegendre}[2]{\big(\tfrac{#1}{#2}\big)}
\newcommand{\odd}{\mathrm{odd}}
\newcommand{\ndiv}{\centernot|}
\newcommand{\ord}{\operatorname{ord}}
\newcommand{\divmin}{\operatorname{div-min}}
\newcommand{\smallmat}[4]{\begin{psmallmatrix}#1&#2\\#3&#4\end{psmallmatrix}}
\newcommand{\smallvect}[2]{\begin{psmallmatrix}#1\\#2\end{psmallmatrix}}

\makeatletter
\newcases{lrdcases}
  {\quad}
  {$\m@th\displaystyle{##}$\hfil}
  {$\m@th\displaystyle{##}$\hfil}
  {\lbrace}
  {\rbrace}
\makeatother

\newtheorem{theorem}{Theorem}[section]
\newtheorem{lemma}[theorem]{Lemma}

\newtheorem{proposition}[theorem]{Proposition}
\newtheorem{conjecture}[theorem]{Conjecture}

\definecolor{lightpink}{rgb}{1.0, 0.71, 0.76}
\definecolor{zaffre}{rgb}{0.0, 0.08, 0.66}
\hypersetup{colorlinks=true,allcolors=zaffre}

\usepackage[nocompress]{cite}

\numberwithin{equation}{section}

\usepackage{booktabs}

\newcommand{\s}{\textit{s}}
\newcommand{\ns}{\textit{ns}}
\newcommand{\Cs}{C_\s}
\newcommand{\Cns}{C_\ns}
\newcommand{\Cr}{C_{\textit{r}}}
\newcommand{\Ns}{N_\s}
\newcommand{\Nns}{N_\ns}
\newcommand{\CM}{\textsc{cm}}
\newcommand{\nonCM}{{\textrm{non-}\textsc{cm}}}
\newcommand{\Excep}{\operatorname{Excep}}

\usepackage{enumitem}
\setenumerate{label={\normalfont(\arabic*)}}

\usepackage{filecontents}
\usepackage{pgfplotstable}
\usetikzlibrary{calc,shadings,patterns}

\title[Cartan Images and $\ell$-Torsion Points of Elliptic Curves]{Cartan Images and $\ell$-Torsion Points of Elliptic Curves with Rational $j$-Invariant}
\author{Oron Y. Propp}
\address{Department of Mathematics, Massachusetts Institute of Technology, Cambridge, MA 02139, USA}
\email{opropp@mit.edu}

\begin{document}

\begin{abstract}
Let $\ell$ be an odd prime and $d$ a positive integer. We determine when there exists a degree-$d$ number field $K$ and an elliptic curve $E/K$ with $j(E)\in\Q\setminus\{0,1728\}$ for which $E(K)_\tor$ contains a point of order $\ell$. We also determine when there exists such a pair $(K,E)$ for which the image of the associated mod-$\ell$ Galois representation is contained in a Cartan subgroup or its normalizer, conditionally on a conjecture of Sutherland. We do the same under the stronger assumption that $E$ is defined over $\Q$.
\end{abstract}

\maketitle

\section{Introduction}

Let~$E$ be an elliptic curve over a number field~$K$ with algebraic closure~$\Kbar$. For an odd prime~$\ell$, let~$E[\ell]$ denote the $\ell$-torsion subgroup of~$E(\Kbar)$, which we recall is a free rank-$2$ $\Z/\ell\Z$-module. The absolute Galois group~$\Gal(\Kbar/K)$ acts on~$E[\ell]$ via its action on the coordinates of its points, which induces a continuous Galois representation
\begin{equation*}
\rho_{E,\ell}\colon\Gal(\Kbar/K)\to\Aut(E[\ell])\simeq\GL_2(\Z/\ell\Z),
\end{equation*}
where the isomorphism with~$\GL_2(\Z/\ell\Z)$ depends on a choice of a $\Z/\ell\Z$-basis of~$E[\ell]$. Thus, we may regard the image of~$\rho_{E,\ell}$ as a subgroup of $\GL_2(\Z/\ell\Z)$ determined up to conjugacy. Moreover, $\rho_{E,\ell}$ factors through the finite Galois group~$\Gal(K(E[\ell])/K)$, where the $\ell$-torsion field~$K(E[\ell])$ is obtained by adjoining to $K$ the coordinates of all $\ell$-torsion points of~$E$.


The problem of identifying the possible images of these Galois representations is closely related to that of determining the possible torsion subgroups and isogenies admitted by elliptic curves over a particular number field. Mazur famously classified the possible torsion subgroups~$E(\Q)_\tor$ in \cite{mazur1977} and the possible $\ell$-isogenies of an elliptic curve~$E/\Q$ in \cite{mazur1978}. Kamienny, Kenku and Momose generalized Mazur's results on torsion subgroups to quadratic number fields in \cite{kamienny1992,kenku1988}, and though no complete characterization for higher-degree number fields is known, there has been recent progress towards characterizing the cubic case \cite{jeon2011,najman2012a,jeon2016,bruin2016,wang2015,najman2016}, the quartic case \cite{jeon2013,najman2012b,gonzalezjimenez2016a,chou2016}, and the quintic case \cite{gonzalezjimenez2016b,clark2014}. In particular, the set of torsion subgroups that arise for infinitely many $\Qbar$-isomorphism classes of elliptic curves defined over number fields of degree $d$ has been determined for $d=3,4,5,6$ \cite{jeon2004,jeon2006,derickx}. The growth of torsion subgroups of elliptic curves upon base change was further studied in \cite{gonzalezjimenez2016c,lozanorobledo2015a}. Using Galois representations, Lozano-Robledo has explicitly characterized the set of primes~$\ell$ for which an elliptic curve~$E/\Q$ has a point of order~$\ell$ over a number field of degree at most~$d$, assuming a positive answer to Serre's uniformity problem \cite{lozanorobledo2013}. The same was done by Najman for cyclic $\ell$-isogenies of elliptic curves with rational $j$-invariant \cite{najman2015}.

We aim to give a more precise characterization of the degrees of number fields over which elliptic curves exhibit certain properties deducible from the images of their Galois representations. Recently, Sutherland developed an algorithm to efficiently compute Galois images of elliptic curves, and conjectured the following stronger version of Serre's uniformity problem in \cite[Conj.~1.1]{sutherland2016} based on data from some 140 million elliptic curves, including all curves of conductor up to 400,000 \cite{cremonatable}:
\begin{conjecture}[Sutherland]
\label{conj:sutherland}
Let~$E/\Q$ be an elliptic curve without complex multiplication and let $\ell$ be a prime. Then the image of~$\rho_{E,\ell}$ is either equal to~$\GL_2(\Z/\ell\Z)$ or conjugate to one of the~$63$ exceptional groups listed in Table~\ref{tab:exim} (see also \cite[Tables~3,4]{sutherland2016}).
\end{conjecture}

\begin{table}
\begin{tabular}{lllclll}
\toprule
$\ell$&Index&Generators&\qquad\qquad&$\ell$&Index&Generators\\
\midrule
\multirow{3}{*}{2} & 6 & $ $ & & \multirow{9}{*}{7} & 48 & $\smallmat{4}{0}{0}{6},\smallmat{1}{1}{0}{1}$ \\
 & 3 & $\smallmat{1}{1}{0}{1}$ & &  & 28 & $\smallmat{0}{6}{1}{0},\smallmat{3}{0}{0}{5},\smallmat{1}{0}{0}{3}$ \\
 & 2 & $\smallmat{0}{1}{1}{1}$ & &  & 24 & $\smallmat{6}{0}{0}{6},\smallmat{1}{0}{0}{3},\smallmat{1}{1}{0}{1}$ \\\cmidrule{1-3}
\multirow{7}{*}{3} & 24 & $\smallmat{1}{0}{0}{2}$ & & & 24 & $\smallmat{6}{0}{0}{6},\smallmat{3}{0}{0}{1},\smallmat{1}{1}{0}{1}$ \\
 & 12 & $\smallmat{2}{0}{0}{2},\smallmat{1}{0}{0}{2}$ & &  & 24 & $\smallmat{6}{0}{0}{6},\smallmat{2}{0}{0}{5},\smallmat{1}{1}{0}{1}$ \\
& 8 & $\smallmat{1}{0}{0}{2},\smallmat{1}{1}{0}{1}$ & & & 21 & $\smallmat{1}{0}{0}{6},\smallmat{2}{5}{4}{2}$ \\
 & 8 & $\smallmat{2}{0}{0}{1},\smallmat{1}{1}{0}{1}$ & & & 16 & $\smallmat{2}{0}{0}{4},\smallmat{1}{0}{0}{3},\smallmat{1}{1}{0}{1}$ \\
 & 6 & $\smallmat{2}{0}{0}{2},\smallmat{0}{2}{1}{0},\smallmat{1}{0}{0}{2}$ & &  & 16 & $\smallmat{2}{0}{0}{4},\smallmat{3}{0}{0}{1},\smallmat{1}{1}{0}{1}$ \\
 & 4 & $\smallmat{2}{0}{0}{2},\smallmat{1}{0}{0}{2},\smallmat{1}{1}{0}{1}$ & &  & 8 & $\smallmat{3}{0}{0}{5},\smallmat{1}{0}{0}{3},\smallmat{1}{1}{0}{1}$ \\\cmidrule{5-7}
 & 3 & $\smallmat{1}{0}{0}{2},\smallmat{2}{1}{2}{2}$ & & \multirow{7}{*}{11} & 120 & $\smallmat{4}{0}{0}{6},\smallmat{1}{1}{0}{1}$ \\\cmidrule{1-3}
\multirow{15}{*}{5} & 120 & $\smallmat{1}{0}{0}{2}$ & & & 120 & $\smallmat{6}{0}{0}{4},\smallmat{1}{1}{0}{1}$ \\
 & 120 & $\smallmat{3}{0}{0}{4}$ & & & 120 & $\smallmat{5}{0}{0}{7},\smallmat{1}{1}{0}{1}$ \\
 & 60 & $\smallmat{4}{0}{0}{4},\smallmat{1}{0}{0}{2}$ & & & 120 & $\smallmat{7}{0}{0}{5},\smallmat{1}{1}{0}{1}$ \\
 & 30 & $\smallmat{2}{0}{0}{3},\smallmat{0}{1}{3}{0}$ & & & 60 & $\smallmat{10}{0}{0}{10},\smallmat{4}{0}{0}{6},\smallmat{1}{1}{0}{1}$ \\
 & 30 & $\smallmat{2}{0}{0}{3},\smallmat{1}{0}{0}{2}$ & & & 60 & $\smallmat{10}{0}{0}{10},\smallmat{5}{0}{0}{7},\smallmat{1}{1}{0}{1}$ \\
 & 24 & $\smallmat{1}{0}{0}{2},\smallmat{1}{1}{0}{1}$ & & & 55 & $\smallmat{1}{0}{0}{10},\smallmat{3}{5}{8}{3}$ \\\cmidrule{5-7}
 & 24 & $\smallmat{2}{0}{0}{1},\smallmat{1}{1}{0}{1}$ & & \multirow{11}{*}{13} & 91 & $\smallmat{3}{0}{12}{9},\smallmat{2}{0}{0}{2},\smallmat{9}{5}{0}{6}$ \\
 & 24 & $\smallmat{4}{0}{0}{3},\smallmat{1}{1}{0}{1}$ & &  & 56 & $\smallmat{3}{0}{0}{9},\smallmat{1}{0}{0}{2},\smallmat{1}{1}{0}{1}$ \\
 & 24 & $\smallmat{3}{0}{0}{4},\smallmat{1}{1}{0}{1}$ & &  & 56 & $\smallmat{3}{0}{0}{9},\smallmat{2}{0}{0}{1},\smallmat{1}{1}{0}{1}$ \\
 & 15 & $\smallmat{0}{4}{1}{0},\smallmat{2}{0}{0}{3},\smallmat{1}{0}{0}{2}$ & &  & 56 & $\smallmat{3}{0}{0}{9},\smallmat{4}{0}{0}{7},\smallmat{1}{1}{0}{1}$ \\
 & 12 & $\smallmat{4}{0}{0}{4},\smallmat{1}{0}{0}{2},\smallmat{1}{1}{0}{1}$ & &  & 56 & $\smallmat{3}{0}{0}{9},\smallmat{7}{0}{0}{4},\smallmat{1}{1}{0}{1}$ \\
 & 12 & $\smallmat{4}{0}{0}{4},\smallmat{2}{0}{0}{1},\smallmat{1}{1}{0}{1}$ & &  & 42 & $\smallmat{5}{0}{0}{8},\smallmat{1}{0}{0}{2},\smallmat{1}{1}{0}{1}$ \\
 & 10 & $\smallmat{1}{0}{0}{4},\smallmat{2}{3}{4}{2}$ & &  & 42 & $\smallmat{5}{0}{0}{8},\smallmat{2}{0}{0}{1},\smallmat{1}{1}{0}{1}$ \\
 & 6 & $\smallmat{2}{0}{0}{3},\smallmat{1}{0}{0}{2},\smallmat{1}{1}{0}{1}$ & &  & 42 & $\smallmat{5}{0}{0}{8},\smallmat{4}{0}{0}{7},\smallmat{1}{1}{0}{1}$ \\
 & 5 & $\smallmat{0}{3}{3}{4},\smallmat{2}{0}{0}{2},\smallmat{3}{0}{4}{4}$ & &  & 28 & $\smallmat{4}{0}{0}{10},\smallmat{1}{0}{0}{2},\smallmat{1}{1}{0}{1}$ \\\cmidrule{1-3}
\multirow{7}{*}{7} & 112 & $\smallmat{2}{0}{0}{4},\smallmat{0}{1}{4}{0}$ & &  & 28 & $\smallmat{4}{0}{0}{10},\smallmat{2}{0}{0}{1},\smallmat{1}{1}{0}{1}$ \\
 & 56 & $\smallmat{3}{0}{0}{5},\smallmat{0}{1}{4}{0}$ & & & 14 & $\smallmat{2}{0}{0}{7},\smallmat{1}{0}{0}{2},\smallmat{1}{1}{0}{1}$ \\\cmidrule{5-7}
 & 48 & $\smallmat{1}{0}{0}{3},\smallmat{1}{1}{0}{1}$ & & \multirow{2}{*}{17} & 72 & $\smallmat{4}{0}{0}{13},\smallmat{2}{0}{0}{10},\smallmat{1}{1}{0}{1}$ \\
 & 48 & $\smallmat{3}{0}{0}{1},\smallmat{1}{1}{0}{1}$ & &  & 72 & $\smallmat{4}{0}{0}{13},\smallmat{6}{0}{0}{9},\smallmat{1}{1}{0}{1}$ \\\cmidrule{5-7}
 & 48 & $\smallmat{2}{0}{0}{5},\smallmat{1}{1}{0}{1}$ & & \multirow{2}{*}{37} & 114 & $\smallmat{8}{0}{0}{14},\smallmat{1}{0}{0}{2},\smallmat{1}{1}{0}{1}$ \\
 & 48 & $\smallmat{5}{0}{0}{2},\smallmat{1}{1}{0}{1}$ & &  & 114 & $\smallmat{8}{0}{0}{14},\smallmat{2}{0}{0}{1},\smallmat{1}{1}{0}{1}$ \\
 & 48 & $\smallmat{6}{0}{0}{4},\smallmat{1}{1}{0}{1}$ & & \\
\bottomrule
\end{tabular}
\caption{Known exceptional images of $\rho_{E,\ell}$ for non-CM elliptic curves $E/\Q$ (for $\ell\le 11$, the list is complete) \cite[Tables~3,4]{sutherland2016}. The first column lists the prime $\ell$; the second column lists the index of the subgroup $G$ in $\GL_2(\Z/\ell\Z)$; and the third lists generators for $G$, as matrices acting on column vectors from the left.}
\label{tab:exim}
\end{table}

Combining this conjecture with the precise characterization of Galois images for elliptic curves over~$\Q$ with complex multiplication (CM) given in \cite[Thm.~1.14-1.16]{zywina2015}, we are able to deduce the possible subgroups of~$\GL_2(\Z/\ell\Z)$ arising as images of~$\rho_{E,\ell}$ for base changes of elliptic curves~$E/\Q$ to a number field~$K$, and for elliptic curves~$E/K$ with $j(E)\in\Q$, assuming $j(E)\ne 0,1728$.\footnote{For $\ell=2$, all conjugacy classes of subgroups of $\GL_2(\Z/\ell\Z)$ occur as the image of $\rho_{E,\ell}$ for some elliptic curve $E/\Q$, so we ignore this case henceforth.} Thus, all results in this paper hold unconditionally for elliptic curves with CM.

Let~$\Cs$ and~$\Cns$ denote the split and nonsplit Cartan subgroups of~$\GL_2(\Z/\ell\Z)$, and~$\Ns$ and~$\Nns$ their normalizers. Also, let $Z$ denote the center of $\GL_2(\Z/\ell\Z)$, that is, the subgroup of scalar matrices, and let
\begin{equation}
\label{eqn:pq}
\begin{split}
\mathcal{P}&:=\{\ell:\tLegendre{-D}{\ell}=1\text{ for all }D\in\{3,4,7,8,11,19,43,67,163\}\}=\{15073,18313,38833,\ldots\},\\
\mathcal{Q}&:=\{\ell:\tLegendre{-D}{\ell}=-1\text{ for all }D\in\{3,4,7,8,11,19,43,67,163\}\}=\{3167,8543,14423,\ldots\}.
\end{split}
\end{equation}
Let $M$ be either $Z$, $\Cs$, $\Cns$, $\Ns$, or $\Nns$. We say that a subgroup $G$ of $\GL_2(\Z/\ell\Z)$ \emph{belongs to} $M$ if it is conjugate to a subgroup of $M$, but not to a subgroup of any other of these five subgroups contained in $M$. For example, if $M$ is $Z$, this amounts to requiring that $G$ consists only of scalars; if $M$ is $\Ns$, then $G$ must be conjugate to a subgroup of $\Ns$ but not to any subgroup of $\Cs$. The inclusions among these five subgroups are depicted in the following diagram on the left:
\begin{equation}
\label{eqn:subincl}
\begin{tikzcd}[column sep=small]
\Ns\arrow[d,dash]&&\Nns\arrow[d,dash]\\
\Cs\arrow[dr,dash]&&\Cns\arrow[dl,dash]\\
&Z,&
\end{tikzcd}
\qquad\qquad\qquad
\begin{tikzcd}[column sep=0pt]
&&X_Z\arrow{dl}\arrow{dr}&&\\
X_\s,\arrow{d}&(X_\s)_{\Q(\sqrt{\pm\ell})}&&(X_\ns)_{\Q(\sqrt{\pm\ell})},&X_\ns\arrow{d}\\
X_\s^+&&&&X_\ns^+.
\end{tikzcd}
\end{equation}

Now, to any subgroup $M$ of $\GL_2(\Z/\ell\Z)$ containing $-I$, we can associate a modular curve $X_M$ defined over the field $K_M:=\Q(\zeta_\ell)^{\det M}$; thus, $X_\s:=X_{\Cs}$, $X_\ns:=X_{\Cns}$, $X_\s^+:=X_{\Ns}$, and $X_\ns^+:=X_{\Nns}$ are defined over $\Q$, whereas $X_Z$ is defined over $\Q(\sqrt{\pm\ell})\subseteq\Q(\zeta_\ell)$. The curve $X_M$ is smooth, projective, and geometrically irreducible, and comes with a natural morphism of $K_M$-schemes
\begin{equation*}
\pi_M\colon X_M\longrightarrow X(1)_{K_M}=\operatorname{Spec}K_M[j]\cup\{\infty\}=\mathbb{P}^1_{K_M}
\end{equation*}
with the property that, for an elliptic curve $E/K_M$ with $j(E)\ne 0,1728$, the image $\rho_{E,\ell}(\Gal(\overline{K_M}/K_M))$ is conjugate to a subgroup of $M$ in $\GL_2(\Z/\ell\Z)$ if and only if $j(E)\in K_M$ is the image of some $K_M$-point of $X_M$ under $\pi_M$ \cite{zywina2015}. The diagram on the right in (\ref{eqn:subincl}) shows the morphisms of these modular curves corresponding to the inclusions in the diagram on the left.

Our first result determines the degrees of number fields $K$ over which the image of the Galois representation attached to an elliptic curve can belong to each of these subgroups $M$, or equivalently, the degrees of $K$-points on the modular curves $X_M$:
\begin{theorem}
\label{thm:1}
Assume Conjecture~\ref{conj:sutherland}. Let $\ell$ be an odd prime, and let $M$ be one of the subgroups $Z$, $\Cs$, $\Cns$, $\Ns$, or $\Nns$ of $\GL_2(\Z/\ell\Z)$. There exists a set $\mathcal{S}_M(\ell)$, identified explicitly in Table~\ref{tab:rml}, such that the following holds: there exists a degree-$d$ number field $K$ and an elliptic curve $E/\Q$ with $j(E)\ne 0,1728$ for which $\rho_{E,\ell}(\Gal(\Kbar/K))$ belongs to $M$ if and only if $s\mid d$ for some $s\in\mathcal{S}_M(\ell)$. The set \(\mathcal{S}_M(\ell)\) satisfies the same condition with $E/\Q$ replaced by an elliptic curve $E/K$ with $j(E)\in\Q\setminus\{0,1728\}$.

Consequently, there exists a non-cuspidal $K$-point on the modular curve $X_M$ that is not in the image of any of the morphisms in (\ref{eqn:subincl}) for some degree-$d$ number field $K$ if and only if $s\mid d$ for some $s\in\mathcal{S}_M(\ell)$.
\end{theorem}

\begin{table}[h]
$$
\begin{tabu}{rlllll}
\toprule
\ell&Z&\Cs&\Cns&\Ns&\Nns\\
\midrule
3&2&1&2&1&1\\
5&4&1&2&1&1\\
7&6,14,16&2,3,7&2&1&1\\
13&24,28,52&2,13&2&1&1\\
17&32,36,272&2,17&2&1&1\\
37&72,76,1332&2,37&2&1&1\\
11,19,43,67,163&2(\ell-1),2\ell,2(\ell+1)&2,\ell&2&1&1\\
\in\mathcal{P}&2(\ell-1)&2&\ell-1&1&(\ell-1)/2\\
\in\mathcal{Q}&2(\ell+1)&\ell+1&2&(\ell+1)/2&1\\
\text{other}&2(\ell-1),2(\ell+1)&2&2&1&1\\
\bottomrule
\end{tabu}
$$
\caption{Values of $\mathcal{S}_M(\ell)$ (see Theorem~\ref{thm:1}).}
\label{tab:rml}
\end{table}

We also identify analogous sets $\mathcal{S}_M^\CM(\ell)$ and $\mathcal{S}_M^\nonCM(\ell)$ for elliptic curves with and without CM, listed in Tables~\ref{tab:rmcml} and \ref{tab:rmnoncml}, respectively. To replace ``belonging'' in the theorem with the ordinary notion of containment, and remove the corresponding restriction related to the morphisms in (\ref{eqn:subincl}), simply take unions of the $\mathcal{S}_M(\ell)$ along the subgroup inclusion lattice in (\ref{eqn:subincl}) (e.g., replace $\mathcal{S}_{\Nns}(\ell)$ with $\mathcal{S}_{\Nns}(\ell)\cup\mathcal{S}_{\Cns}(\ell)\cup\mathcal{S}_{Z}(\ell)$). For interpretations of the latter statement on modular curves in terms of moduli data, see the table on page~2 of \cite{rebolledo}.

\begin{table}
$$
\begin{tabu}{rlllll}
\toprule
\ell&Z&\Cs&\Cns&\Ns&\Nns\\
\midrule
3,7,11,19,43,67,163&2(\ell-1),2\ell,2(\ell+1)&2,\ell&2&1&1\\
\in\mathcal{P}&2(\ell-1)&2&\ell-1&1&(\ell-1)/2\\
\in\mathcal{Q}&2(\ell+1)&\ell+1&2&(\ell+1)/2&1\\
\text{other}&2(\ell-1),2(\ell+1)&2&2&1&1\\
\bottomrule
\end{tabu}
$$
\caption{Values of $\mathcal{S}_{M}^\CM(\ell)$ (see Theorem~\ref{thm:1}).}
\label{tab:rmcml}
\end{table}

\begin{table}
$$
\begin{tabu}{rlllll}
\toprule
\ell&Z&\Cs&\Cns&\Ns&\Nns\\
\midrule
3&2&1&2&1&1\\
5&4&1&2&1&1\\
7&6,14,16&2,3,7&2&1&1\\
11&24,110&11,12&2&6&1\\
13&24,52&6,8,13&12&3,4&6,26\\
17&272&17&272&153&136\\
37&1332&37&1332&703&666\\
\text{other}&\ell(\ell+1)(\ell-1)&\ell(\ell+1)&\ell(\ell-1)&\ell(\ell+1)/2&\ell(\ell-1)/2\\
\bottomrule
\end{tabu}
$$
\caption{Values of $\mathcal{S}_M^\nonCM(\ell)$ (see Theorem~\ref{thm:1}).}
\label{tab:rmnoncml}
\end{table}

Our next result characterizes the degrees of abelian sub-extensions of $\ell$-torsion fields $\Q(E[\ell])$:
\begin{theorem}
\label{thm:2}
Assume Conjecture~\ref{conj:sutherland}. Given an odd prime $\ell$, there exists a  set $\mathcal{K}(\ell)$, identified explicitly in (\ref{eqn:vwl}), such that the following holds: there exists a degree-$d$ number field $K$ contained in the $\ell$-torsion field of an elliptic curve $E/\Q$ with $j(E)\ne 0,1728$ for which $\Q(E[\ell])/K$ is an abelian extension precisely when $d\in\mathcal{K}(\ell)$. The set $\mathcal{K}(\ell)$ satisfies the same condition with $E/\Q$ replaced by an elliptic curve $E/K$ with $j(E)\in\Q\setminus\{0,1728\}$.
\end{theorem}

As above, we explicitly identify analogous subsets $\mathcal{K}^\CM(\ell),\mathcal{K}^\nonCM(\ell)\subseteq\mathcal{K}(\ell)$ for elliptic curves with and without CM, listed in (\ref{eqn:vwl}). Note that Theorem~\ref{thm:1} already implies that there exists a degree-$d$ number field $K$ and an elliptic curve $E/\Q$ with $j(E)\ne 0,1728$ for which the extension $\Q(E[\ell])/(K\cap\Q(E[\ell])$ is abelian precisely when $r\mid d$ for some $r\in\mathcal{S}_Z(\ell)\cup\mathcal{S}_{\Cs}(\ell)\cup\mathcal{S}_{\Cns}(\ell)$; as in the statement of Theorem~\ref{thm:1}, $E$ may be replaced by an elliptic curve $E/K$ with $j(E)\in\Q\setminus\{0,1728\}$. However, Theorem~\ref{thm:2} provides a more precise characterization of abelian subfields of $\ell$-torsion fields by identifying when $\Q(E[\ell])/K$ is abelian, rather than $\Q(E[\ell])/(K\cap\Q(E[\ell]))$. Table~\ref{tab:ldabext} illustrates the first result for all pairs $(\ell,d)$ with $\ell\le 97$ and $d\le 45$; a light square indicates that such a number field and elliptic curve exist, and a dark square indicates that they do not \cite{code}. In particular, as shown unconditionally in \cite[Thm.~1.1]{lozanorobledo2015}, $\Q(E[\ell])/\Q$ can only be abelian for $\ell=3,5$, though naturally for any elliptic curve $E/\Q$ with CM there exists a quadratic field $K$ for which $\Q(E[\ell])/K$ is abelian.
\begin{table}
\begin{tikzpicture}[x=1em,y=-1em]
\node at (-0.4,0.3) {$\ell$};
\node at (0.3,-0.4) {$d$};
\draw (1,1) -- (-0.5,-0.5);

\pgfplotstableread[col sep=space,row sep=newline]{collab}{\collab}

\pgfplotstableforeachcolumn\collab\as\col{
  \pgfplotstableforeachcolumnelement{\col}\of\collab\as\colcnt{
    \node[left] at (0.65,\col+1) {\footnotesize$\colcnt$};
  }
}

\pgfplotstableread[col sep=space,row sep=newline]{rowlab}{\rowlab}

\pgfplotstableforeachcolumn\rowlab\as\col{
  \pgfplotstableforeachcolumnelement{\col}\of\rowlab\as\colcnt{
    \node at (\col+1,0) {\footnotesize$\colcnt$};
  }
}

\pgfplotstableread[col sep=space,row sep=newline]{bindat}{\bindat}

\pgfplotstableforeachcolumn\bindat\as\col{
  \ifnum\col<45
  \pgfplotstableforeachcolumnelement{\col}\of\bindat\as\colcnt{
    \ifnum\colcnt=0
      \filldraw [draw=black,fill=BrickRed] (0.5+\col,0.5+\pgfplotstablerow) rectangle (1.5+\col,1.5+\pgfplotstablerow);
    \fi
    \ifnum\colcnt=1
      \filldraw [draw=black,fill=lightpink] (0.5+\col,0.5+\pgfplotstablerow) rectangle (1.5+\col,1.5+\pgfplotstablerow);
    \fi
  }
  \fi
}
\end{tikzpicture}
\caption{Light squares identify pairs $(\ell,d)$ for which there exists a degree-$d$ number field $K$ contained in the $\ell$-torsion field of an elliptic curve $E/\Q$ (or $E/K$) with $j(E)\in\Q\setminus\{0,1728\}$ such that $\Q(E[\ell])/K$ is abelian (see Theorem~\ref{thm:2}).}
\label{tab:ldabext}
\end{table}

Finally, we have the following result, which is a refinement of \cite[Thm.~1.7]{lozanorobledo2013} on degrees of number fields over which elliptic curves with rational $j$-invariant contain torsion points of order $\ell$:
\begin{theorem}
\label{thm:3}
Assume Conjecture~\ref{conj:sutherland}. Given an odd prime $\ell$, there exists a set $\mathcal{T}(\ell)$, identified explicitly in Table~\ref{tab:tjl}, such that the following holds: there exists a degree-$d$ number field $K$ and an elliptic curve $E/K$ with $j(E)\in\Q\setminus\{0,1728\}$ for which $E(K)_\tor$ contains a point of order $\ell$ if and only if $t\mid d$ for some $t\in\mathcal{T}(\ell)$.
\end{theorem}
The analogous result for elliptic curves defined over $\Q$ is proven unconditionally in \cite[Cor.~6.1]{gonzalezjimenez2016c}; our methods can be easily made to give a proof of this case, which we omit, instead reproducing the resultant sets $\mathcal{T}_\Q(\ell)$ in Table~\ref{tab:tjl}. We again identify analogous sets $\mathcal{T}^\CM(\ell)$, $\mathcal{T}_\Q^\CM(\ell)$, $\mathcal{T}^\nonCM(\ell)$, $\mathcal{T}_\Q^\nonCM(\ell)$ for elliptic curves with and without CM, listed with $\mathcal{T}(\ell)$ and $\mathcal{T}_\Q(\ell)$. These results agree with, and in some cases sharpen, recent work on CM elliptic curves \cite{bourdonclarkstankewicz,bourdonpollack,bourdonclark}. For instance, \cite[Thm.~1.5]{bourdonclarkstankewicz} and \cite[Thm.~1.2]{bourdonpollack} imply that there exists a number field of odd degree $d$ over which a CM elliptic curve attains an $\ell$-torsion point if and only if $\ell\equiv 3\bmod 4$ and $d$ is divisible by $h_\ell\cdot\frac{\ell-1}{2}$, where $h_\ell$ denotes the class number of $\Q(\sqrt{-\ell})$ (see also \cite[Thm.~6.2]{bourdonclark}). But by Theorem~\ref{thm:3}, for certain primes $\ell\equiv 3\bmod 4$, all odd degrees $d$ divisible by $h_\ell\cdot\frac{\ell-1}{2}$ have this property even if we restrict to CM elliptic curves $E$ with $j(E)\in\Q\setminus\{0,1728\}$; for other primes $\ell\equiv 3\bmod 4$, no such degrees do. Table~\ref{tab:cm} illustrates this phenomenon for all such $\ell\le 83$. The entries in its final column can also be derived from \cite{bourdonclarkstankewicz,bourdonpollack}: they are ``all'' when $h_\ell=1$ and ``none'' when $h_\ell>1$.

\begin{table}
$$
\begin{tabu}{rlll}
\toprule
\ell&\mathcal{T}(\ell)&\mathcal{T}^\CM(\ell)&\mathcal{T}^\nonCM(\ell)\\
\midrule
3&1&1&1\\
5&1&4&1\\
7&1&3&1\\
11&5&5&5\\
13&2,3&12&2,3\\
17&4&16&4\\
37&6&36&6\\
19,43,67,163&(\ell-1)/2&(\ell-1)/2&(\ell^2-1)/2\\
\in\mathcal{Q}&(\ell^2-1)/2&(\ell^2-1)/2&(\ell^2-1)/2\\
\text{other}&\ell-1&\ell-1&(\ell^2-1)/2\\
\bottomrule
\end{tabu}
$$
\caption{Values of $\mathcal{T}(\ell)$, $\mathcal{T}^\CM(\ell)$, and $\mathcal{T}^\nonCM(\ell)$ (see Theorem~\ref{thm:3}).}
\label{tab:tjl}
\end{table}

\begin{table}
$$
\begin{tabu}{rlll}
\toprule
\ell&\mathcal{T}_\Q(\ell)&\mathcal{T}_\Q^\CM(\ell)&\mathcal{T}_\Q^\nonCM(\ell)\\
\midrule
3&1&1&1\\
5&1&8&1\\
7&1&3&1\\
11&5&5&5\\
13&3,4&24&3,4\\
17&8&32&8\\
37&12&72&12\\
19,43,67,163&(\ell-1)/2&(\ell-1)/2&\ell^2-1\\
\in\mathcal{Q}&\ell^2-1&\ell^2-1&\ell^2-1\\
\text{other}&2(\ell-1)&2(\ell-1)&\ell^2-1\\
\bottomrule
\end{tabu}
$$
\caption{Values of $\mathcal{T}_\Q(\ell)$, $\mathcal{T}_\Q^\CM(\ell)$, and $\mathcal{T}_\Q^\nonCM(\ell)$ (see Theorem~\ref{thm:3}).}
\label{tab:tl}
\end{table}

\begin{table}
$$
\begin{tabu}{rlll}
\toprule
\ell&h_\ell\cdot\frac{\ell-1}{2}&\mathcal{T}^\CM(\ell)&\text{Degrees}\\\midrule
3&1&1&\text{all}\\
7&3&3&\text{all}\\
11&5&5&\text{all}\\
19&9&9&\text{all}\\
23&33&22&\text{none}\\
31&45&30&\text{none}\\
43&21&21&\text{all}\\
47&115&46&\text{none}\\
59&87&58&\text{none}\\
67&33&33&\text{all}\\
71&245&70&\text{none}\\
79&195&78&\text{none}\\
83&123&82&\text{none}\\\bottomrule
\end{tabu}
$$
\caption{Primes $\ell\equiv 3\bmod 4$, along with whether ``all'' or ``none'' of the odd numbers $d$ divisible by $h_\ell\cdot\frac{\ell-1}{2}$ have the property that there exists a degree-$d$ number field over which a CM elliptic curve $E$ with $j(E)\in\Q\setminus\{0,1728\}$ attains an $\ell$-torsion point.}
\label{tab:cm}
\end{table}

\subsection*{Acknowledgements}

The author wishes to thank Andrew Sutherland for supervising this research and providing many helpful suggestions and insights; Filip Najman and \'Alvaro Lozano-Robledo for their comments on an early draft of this paper; and the referees for their comments and careful review. This research was generously supported by \textsc{mit}'s \textsc{urop} program and the Paul E. Gray (1954) Endowed Fund for \textsc{urop}.

\section{Background}
\label{sec:background}

Let $\ell$ be an odd prime, and let $\GL_2(\ell):=\GL_2(\Z/\ell\Z)$ and $\SL_2(\ell):=\SL_2(\Z/\ell\Z)$. The center of $\GL_2(\ell)$ is the subgroup
\begin{equation*}
Z(\ell):=\{\smallmat{x}{0}{0}{x}:x\in\F_\ell^\times\}\simeq\F_\ell^\times,
\end{equation*}
and we let $\PGL_2(\ell):=\GL_2(\ell)/Z(\ell)$. Let
\begin{align*}
\Cs(\ell)&:=\{\smallmat{x}{0}{0}{y}:xy\ne 0\}\le\GL_2(\ell),\\
\Cns(\ell)&:=\{\smallmat{x}{\epsilon y}{y}{x}:(x,y)\ne(0,0)\}\le\GL_2(\ell),
\end{align*}
be the \emph{split Cartan group} and \emph{nonsplit Cartan group}, respectively, for some fixed non-quadratic residue $\epsilon\in\F_\ell^\times$. We refer to their conjugates in $\GL_2(\ell)$ as \emph{split} and \emph{nonsplit Cartan groups}, respectively. Note that we have $\Cs(\ell)\simeq\F_\ell^\times\times\F_\ell^\times$ and $\Cns(\ell)\simeq\F_{\ell^2}^\times$. Both $\Cs(\ell)$ and $\Cns(\ell)$ are index-$2$ subgroups of their normalizers
\begin{equation*}
\Ns(\ell):=\Cs\cup\smallmat{0}{1}{1}{0}\Cs,\qquad\Nns(\ell):=\Cns\cup\smallmat{1}{0}{0}{-1}\Cns.
\end{equation*}
We denote the matrices in $\Ns(\ell)$ and $\Nns(\ell)$ by
\begin{equation*}
\begin{aligned}
\mathcal{D}_\s(x,y)&:=\smallmat{x}{0}{0}{y},\\
\mathcal{D}_\ns(x,y)&:=\smallmat{x}{\epsilon y}{y}{x},
\end{aligned}
\qquad
\begin{aligned}
\mathcal{A}_\s(x,y)&:=\smallmat{0}{x}{y}{0},\\
\mathcal{A}_\ns(x,y)&:=\smallmat{x}{-\epsilon y}{y}{-x}.
\end{aligned}
\end{equation*}
The \emph{Borel group} $B(\ell)\le\GL_2(\ell)$ is the subgroup of upper triangular matrices; we refer to its conjugates in $\GL_2(\ell)$ as \emph{Borel groups}. We refer to $\GL_2(\ell)$, the Borel groups, the Cartan subgroups, and their normalizers collectively as \emph{standard subgroups} of $\GL_2(\ell)$.

The following proposition, originally due to Dickson \cite{dickson1901}, classifies subgroups of $\GL_2(\ell)$ in terms of their images in $\PGL_2(\ell)$:

\begin{proposition}
\label{prop:subs}
Let $\ell$ be an odd prime and $G$ be a subgroup of $\GL_2(\ell)$ with image $H$ in $\PGL_2(\ell)$. If $G$ contains an element of order $\ell$, then either $G$ is conjugate to a subgroup of the Borel group $B(\ell)$, or $\SL_2(\ell)$ is contained in $G$. Otherwise, one of the following holds:
\begin{enumerate}
\item $H$ is cyclic and a conjugate of $G$ lies in $\Cs(\ell)$ or $\Cns(\ell)$;
\item $H$ is dihedral and a conjugate of $G$ lies in $\Ns(\ell)$ or $\Nns(\ell)$, but not in $\Cs(\ell)$ or $\Cns(\ell)$;
\item $H$ is isomorphic to $A_4$, $S_4$, or $A_5$, and no conjugate of $G$ is contained in the normalizer of any Cartan subgroup.
\end{enumerate}
\end{proposition}
\begin{proof}
See \cite[\S2]{serre1972}.
\end{proof}

An integer $n\in S\subseteq\Z_{>0}$ is said to be \emph{div-minimal in $S$} if it is a minimal element of $S$ in the divisibility lattice of integers. A subgroup $H\in S\subseteq\{G:G\le\GL_2(\ell)\}$ is said to be \emph{div-minimal in $S$} if $\#H$ is div-minimal in $\{\#G:G\in S\}$. The terms ``div-maximal,'' ``div-minima,'' and so forth, are defined as expected. The set $S$ is often omitted if it is clear from context. Note that an element of a set $S$ can be both div-minimal and div-maximal.

\section{Proof of Theorems~\ref{thm:1} and \ref{thm:2}}

Let $E/\Q$ be an elliptic curve with $j(E)\ne 0,1728$, and let $K$ be a degree-$d$ number field. The intersection $K\cap\Q(E[\ell])$ with the $\ell$-torsion field of $E$ is a number field of degree $r$ dividing $d$. Conversely, for any degree-$r$ number field $L$ that is a subfield of $\Q(E[\ell])$, and $d$ divisible by $r$, there exists a degree-$d$ number field $L'$ such that $L'\cap\Q(E[\ell])=L$ (e.g., adjoin to $L$ a root of $x^{d/r}-p$ for some prime $p$ unramified in both $L$ and $\Q(E[\ell])$). Since
\begin{equation}
\label{eqn:galisomsub}
\rho_{E,\ell}(\Gal(\Kbar/K))\simeq\Gal(\Q(E[\ell])/(K\cap\Q(E[\ell]))\le\Gal(\Q(E[\ell])/\Q)\simeq\rho_{E,\ell}(\Gal(\Qbar/\Q))
\end{equation}
is an index-$r$ subgroup, it suffices to determine which index-$r$ subgroups of each possible image $G_E(\ell):=\rho_{E,\ell}(\Gal(\Qbar/\Q))$ belong to each standard subgroup in question. That is, we would like to identify the sets $\I(G_E(\ell);M(\ell))$ for each $M\in\{Z,\Cs,\Cns,\Ns,\Nns\}$, where we have let
\begin{equation}
\label{eqn:IGM}
\I(G_1,\ldots,G_n;M):=\{[G_i:H]:H\le G_i\text{ such that }H\text{ belongs to }M\text{, for some }1\le i\le n\}
\end{equation}
for any $G_1,\ldots,G_n\le\GL_2(\ell)$ and standard subgroup $M$.

If $E$ does not have CM, then $G_E(\ell)$ is either $\GL_2(\ell)$ or one of the $63$ exceptional groups listed in Table~\ref{tab:exim} by our assumption of Conjecture~\ref{conj:sutherland}. Let \(\Excep(\ell)\) denote the set of exceptional subgroups of \(\GL_2(\ell)\). The results for each subgroup in Table~\ref{tab:exim} were computed \cite{code,sutherland2016} using Magma \cite{magma}, and the unions of these sets of indices over each \(\Excep(\ell)\) are the sets $\I(\Excep(\ell);M(\ell))$ appearing in Table~\ref{tab:eml}. The analogous sets $\I(\GL_2(\ell);M(\ell))$ are determined in \S\ref{subsec:f}, and appear in Table~\ref{tab:setml}. There, and henceforth, we define
\begin{equation}
\label{eqn:delta}
\Delta(r\in L\mid s\in U):=\Delta(L,U):=\{n\in\Z_{>0}:r\mid n\mid s\text{ for some }r\in L,s\in U\}
\end{equation}
for any nonempty $L,U\subseteq\Z_{>0}$; we also let
\begin{equation}
\label{eqn:piodd}
\Pi(n):=\{p\mid n\},\qquad\Pi_\odd(n):=\{p\mid n:p>2\},
\end{equation}
and $\nu_2$ denote the usual $2$-adic valuation.

Now suppose $E$ has CM. Following \cite{zywina2015}, define subgroups
\begin{align*}
G(\ell)&:=\{\smallmat{a}{b}{0}{\pm a}:a\in\F_\ell^\times,b\in\F_\ell\},\\
H_1(\ell)&:=\{\smallmat{a}{b}{0}{\pm a}:a\in(\F_\ell^\times)^2,b\in\F_\ell\},\\
H_2(\ell)&:=\{\smallmat{\pm a}{b}{0}{a}:a\in(\F_\ell^\times)^2,b\in\F_\ell\},
\end{align*}
and recall the definitions of (\ref{eqn:pq}). By \cite[Prop.~1.14]{zywina2015}, if $\ell\in\mathcal{P}$ (resp.\ $\ell\in\mathcal{Q}$) then $G_E(\ell)$ is conjugate to $\Ns(\ell)$ (resp.\ $\Nns(\ell)$); if $\ell\in\{3,7,11,19,43,67,163\}$ (i.e., the set of odd primes $\ell$ such that the imaginary quadratic field $\Q(\sqrt{-\ell})$ has class number $1$) then $G_E(\ell)$ is conjugate to either $\Ns(\ell)$, $\Nns(\ell)$, $G(\ell)$, $H_1(\ell)$, or $H_2(\ell)$, and moreover each of these cases occurs for some such $E$ (as a simple Magma computation \cite{code} shows that $\big\{\tLegendre{-D}{\ell}:D\in\{3,4,7,8,11,19,43,67,163\}\big\}=\{-1,0,1\}$ for all $\ell\in\{3,7,11,19,43,67,163\}$); and otherwise $G_E(\ell)$ is conjugate to either $\Ns(\ell)$ or $\Nns(\ell)$, with each case occurring for some such $E$. The sets $\I(\Ns(\ell);M(\ell))$, $\I(\Nns(\ell);M(\ell))$, and $\I(G(\ell),H_1(\ell),H_2(\ell);M(\ell))$ are determined in \S\ref{subsec:s}, \S\ref{subsec:n}, and \S\ref{subsec:b}, respectively, and appear in Table~\ref{tab:setml}.

Now let $K$ be a degree-$d$ number field, and $E/K$ be an elliptic curve with $j(E)\in\Q\setminus\{0,1728\}$. Then $E$ is either a base change of an elliptic curve over $\Q$, or a quadratic twist of such a curve. By \cite[Cor.~5.25]{sutherland2016}, the images $\rho_{E,\ell}(\Gal(\Kbar/K))$ which may arise in this case are precisely the groups $G$ arising for base changes of elliptic curves $E/\Q$, along with their \emph{twists}: these are the group $\angles{G,-1}$ and its index-$2$ subgroups that do not contain $-1$. The following lemma shows that a subgroup of $\GL_2(\ell)$ and its twists always belong to the same subgroup, and therefore all results above hold when $E/\Q$ is replaced by an elliptic curve $E/K$ with $j(E)\in\Q\setminus\{0,1728\}$.

\begin{lemma}
\label{lem:twistsamemaxsubgp}
Let $G\le\GL_2(\ell)$, let $H$ be a twist of $G$, and let $M$ be $Z$, $\Cs$, $\Cns$, $\Ns$, or $\Nns$. Then $G$ belongs to $M$ if and only if $H$ belongs to $M$.
\end{lemma}
\begin{proof}
Let $H$ be an index-$2$ subgroup of $\angles{G,-1}$ not containing $-1$, and suppose $H\le M$. Then $H$ is an index-$2$ subgroup of $\angles{H,-1}\simeq H\times\{\pm 1\}$, so since $\angles{H,-1}\subseteq\angles{G,-1}$, it follows that $\angles{H,-1}=\angles{G,-1}$. But $-1\in M$ (as is true of all standard subgroups of $\GL_2(\ell)$), hence $G\le\angles{H,-1}\le M$. Conversely, if $G\le M$, then $H\le\angles{G,-1}\le M$ as $-1\in M$, so altogether $H\le M$ if and only if $G\le M$. Since the above argument respects conjugacy, the conclusion follows by definition. The same proof applies if $H=\angles{G,-1}$.
\end{proof}

Finally, let
\begin{gather*}
\mathcal{E}_M^\CM(\ell):=\begin{lrdcases}\I(\Ns(\ell);M(\ell))&\text{if }\ell\in\mathcal{P},\\
\I(\Nns(\ell);M(\ell))&\text{if }\ell\in\mathcal{Q},\\
\I(\Ns(\ell),\Nns(\ell),G(\ell),H_1(\ell),H_2(\ell);M(\ell))&\text{if }\ell\in\{3,7,11,19,43,67,163\},\\
\I(\Ns(\ell),\Nns(\ell);M(\ell))&\text{otherwise}.
\end{lrdcases},\\
\mathcal{E}_M^\nonCM(\ell):=\I(\Excep(\ell);M(\ell))\cup\I(\GL_2(\ell);M(\ell)),\qquad\mathcal{E}_M(\ell):=\mathcal{E}_M^\CM\cup\mathcal{E}_M^\nonCM.
\end{gather*}
Altogether, we have proven the following result:
\begin{theorem}
\label{thm:main}
Assume Conjecture~\ref{conj:sutherland}. Let $\ell$ be an odd prime, and let $M\in\{\Cs,\Cns,\Ns,\Nns\}$. There exists a degree-$d$ number field $K$ and an elliptic curve $E/\Q$ with $K$ a subfield of $\Q(E[\ell])$ such that $\rho_{E,\ell}(\Gal(\Kbar/K)))$ belongs to $M(\ell)$ if and only if $d\in\mathcal{E}_M(\ell)$. The same is true if $E$ is replaced by an elliptic curve $E/K$ with $j(E)\in\Q\setminus\{0,1728\}$.
\end{theorem}
Naturally, to restrict to an elliptic curve with or without CM, we may simply replace $\mathcal{E}_M(\ell)$ by $\mathcal{E}_M^\CM(\ell)$ or $\mathcal{E}_M^\nonCM(\ell)$, respectively. To obtain Theorem~\ref{thm:1}, it remains to choose the div-minimal elements of $\mathcal{E}_M(\ell)$, by (\ref{eqn:galisomsub}). Doing so yields the sets $\mathcal{S}_M(\ell)$ in Table~\ref{tab:rml}; the sets $\mathcal{S}_M^\CM(\ell)$ and $\mathcal{S}_M^\nonCM(\ell)$ are obtained similarly.

For Theorem~\ref{thm:2}, note that the Cartan subgroups are the maximal abelian subgroups of $\GL_2(\ell)$ not containing an element of order $\ell$; by \cite[Lem.~3.3]{sutherland2016}, the maximal abelian subgroup of $\GL_2(\ell)$ containing an element of order $\ell$ is given up to conjugacy by the \emph{ramified Cartan group}
\begin{equation*}
\Cr(\ell):=\{\smallmat{a}{b}{0}{a}:a\in\F_\ell^\times,b\in\F_\ell\}.
\end{equation*}
Thus, it remains to determine the sets $\I(G_E(\ell);\Cr(\ell))$ for each $G_E$ appearing in Table~\ref{tab:setml}; note that since $Z(\ell)\subseteq \Cr(\ell)$, we say that a subgroup belongs to $\Cr(\ell)$ if it is conjugate to a subgroup of $\Cr(\ell)$, but not to a subgroup of $Z(\ell)$. These sets appear in Tables~\ref{tab:eml} and \ref{tab:setml} and are determined in Lemma~\ref{lem:A} below. It then follows from (\ref{eqn:galisomsub}) that
\begin{equation}
\label{eqn:vwl}
\begin{split}
\mathcal{K}(\ell)&=\mathcal{E}_Z(\ell)\cup\mathcal{E}_{\Cs}(\ell)\cup\mathcal{E}_{\Cns}(\ell)\cup\mathcal{E}_{\Cr}(\ell),\\
\mathcal{K}^\CM(\ell)&=\mathcal{E}_Z^\CM(\ell)\cup\mathcal{E}_{\Cs}^\CM(\ell)\cup\mathcal{E}_{\Cns}^\CM(\ell)\cup\mathcal{E}_{\Cr}^\CM(\ell),\\
\mathcal{K}^\nonCM(\ell)&=\mathcal{E}_Z^\nonCM(\ell)\cup\mathcal{E}_{\Cs}^\nonCM(\ell)\cup\mathcal{E}_{\Cns}^\nonCM(\ell)\cup\mathcal{E}_{\Cr}^\nonCM(\ell),
\end{split}
\end{equation}
which is the desired explicit identification.

\begin{lemma}
\label{lem:A}
The set of indices of subgroups $G$ of $G_E(\ell)$ such that $G$ belongs to $\Cr(\ell)$ is the set of positive integers divisible by $\ell^2-1$ (resp. $2$) and dividing $(\ell+1)(\ell-1)^2$ (resp. $2(\ell-1)$) if $G_E$ is $\GL_2$ (resp. $G$, $H_1$, or $H_2$), and the empty set if $G_E$ is $\Ns$ or $\Nns$. 
\end{lemma}
\begin{proof}
A subgroup of $\GL_2(\ell)$ belongs to $\Cr(\ell)$ if and only if it is contained in $\Cr(\ell)$ and contains $\angles{\smallmat{1}{1}{0}{1}}$; the set of indices of such subgroups is evidently
\begin{equation*}
\Delta([\GL_2(\ell):\Cr(\ell)]\mid[\GL_2(\ell):\angles{\smallmat{1}{1}{0}{1}}])=\Delta(\#\GL_2(\ell)/\ell(\ell-1)\mid \#\GL_2(\ell)/\ell),
\end{equation*}
as desired.

Since $\Cr(\ell)\subseteq G(\ell)$, we have
\begin{equation*}
\I(G(\ell);\Cr(\ell))=\Delta((\ell^2-1)/[\GL_2(\ell):G(\ell)]\mid(\ell+1)(\ell-1)^2/[\GL_2(\ell):G(\ell)])=\Delta(2\mid 2(\ell-1)).
\end{equation*}
The computations for $H_1(\ell)$ and $H_2(\ell)$ are analogous, and taking unions gives the result when $G_E$ is allowed to be any of these three subgroups.

For the cases where $G_E(\ell)$ is $\Ns(\ell)$ or $\Nns(\ell)$, it suffices to note that neither of these subgroups contains an element of order $\ell$.
\end{proof}

\begin{table}
\begin{gather*}
\begin{tabu}{rll}
\toprule
\ell&Z(\ell)&\Cs(\ell)\\
\midrule
3&\Delta(2\mid 12,16)&\Delta(1\mid 6,8)\\
5&\Delta(4\mid 80,96)&\Delta(1\mid 40,48)\setminus\{3\}\\
7&\Delta(6,14,16\mid 72,96,252)&\Delta(2,3,7\mid 36,48,84,126)\\
11&\Delta(24,110\mid 220,240)&\Delta(11,12\mid 44,110,120)\\
13&\Delta(24,52\mid 288,1872)&\Delta(6,8,13\mid 96,144,624,936)\\
17&\Delta(272\mid 1088)&\Delta(17\mid 544)\\
37&\Delta(1332\mid 15984)&\Delta(37\mid 5328,7992)\\
\text{other}&\varnothing&\varnothing
\end{tabu}\\
\begin{tabu}{rllll}
\toprule
\ell&\Cns(\ell)&\Ns(\ell)&\Nns(\ell)&\Cr(\ell)\\
\midrule
3&\Delta(2\mid 4)&\Delta(1\mid 4)&\Delta(1\mid 6,8)&\Delta(2\mid 4)\\
5&\Delta(2\mid 12,32)&\Delta(1\mid 12)&\Delta(1\mid 40,48)\setminus\{5\}&\Delta(4\mid 16)\\
7&\Delta(2\mid 18,24)&\Delta(1\mid 18,24)&\Delta(1\mid 36,48,126)&\Delta(2\mid 26)\\
11&\Delta(2\mid 60,80)&\Delta(6\mid 60)&\Delta(1\mid 110,120)\setminus\{11,22\}&\Delta(10\mid 20)\\
13&\Delta(12\mid 36)&\Delta(3,4\mid 36)&\Delta(6,26\mid 144,936)&\Delta(4\mid 144)\\
17&\varnothing&\varnothing&\Delta(136\mid 544)&\Delta(16\mid 64)\\
37&\varnothing&\varnothing&\Delta(666\mid 7992)&\Delta(36\mid 432)\\
\text{other}&\varnothing&\varnothing&\varnothing&\varnothing\\
\bottomrule
\end{tabu}
\end{gather*}
\caption{The set $\I(\Excep(\ell);M(\ell))$ of indices of subgroups $G$ of some exceptional image $G_E(\ell)=\rho_{E,\ell}(\Gal(\Qbar/\Q))$ listed in Table~\ref{tab:exim} such that $G$ belongs to $M(\ell)$, for the standard subgroups $M=Z,\Cs,\Cns,\Ns,\Nns,\Cr$. Regarding notation, see (\ref{eqn:delta}).}
\label{tab:eml}
\end{table}

\begin{table}
\begin{tabular}{rrl}
\toprule
$G_E$&$M$&$\I(G_E(\ell);M(\ell))$\\
\midrule
\multirow{6}{*}{$\GL_2$}&$Z$&$\Delta(\ell(\ell+1)(\ell-1)\mid\ell(\ell+1)(\ell-1)^2)$\\
&$\Cs$&$\Delta(\{\ell(\ell+1)\},\{\ell(\ell+1)(\ell-1)^2/q:q\in\Pi(\ell-1)\})$\\
&$\Cns$&$\Delta(\{\ell(\ell-1)\},\{\ell(\ell+1)(\ell-1)^2/q:q\in\Pi_\odd(\ell+1)\cup\{2^{\nu_2(\ell-1)+1}\}\})$\\
&$\Ns$&$\Delta(\{\ell(\ell+1)/2\},\{\ell(\ell+1)(\ell-1)^2/2q:q\in\Pi_\odd(\ell-1)\cup\{4\text{ if }4\mid(\ell-1),\text{ else }2\}\})$\\
&$\Nns$&$\Delta(\ell(\ell-1)/2\mid\ell(\ell+1)(\ell-1)^2/2)$\\
&$\Cr$&$\Delta(\ell^2-1\mid(\ell+1)(\ell-1)^2)$\\
\midrule
\multirow{6}{*}{$\Ns$}&$Z$&$\Delta(2(\ell-1)\mid 2(\ell-1)^2)$\\
&$\Cs$&$\Delta(\{2\},\{2(\ell-1)^2/q:q\in\Pi(\ell-1)\})$\\
&$\Cns$&$\Delta(\ell-1\mid (\ell-1)^2/2^{\nu_2(\ell-1)})$\\
&$\Ns$&$\Delta(\{1\},\{(\ell-1)^2/q:q\in\Pi_\odd(\ell-1)\cup\{4\text{ if }4\mid(\ell-1),\text{ else }2\}\})$\\
&$\Nns$&$\Delta((\ell-1)/2\mid(\ell-1)^2)$\\
&$\Cr$&$\varnothing$\\
\midrule
\multirow{6}{*}{$\Nns$}&$Z$&$\Delta(2(\ell+1)\mid 2(\ell^2-1))$\\
&$\Cs$&$\Delta(\ell+1\mid\ell^2-1)$\\
&$\Cns$&$\Delta(\{2\},\{2(\ell^2-1)/q:q\in\Pi_\odd(\ell+1)\cup\{2^{\nu_2(\ell-1)+1}\}\})$\\
&$\Ns$&$\Delta((\ell+1)/2\mid(\ell^2-1)/2^{\nu_2(\ell-1)})$\\
&$\Nns$&$\Delta(1\mid(\ell^2-1))$\\
&$\Cr$&$\varnothing$\\
\midrule
\multirow{6}{*}{$G,H_1,H_2$}&$Z$&$\Delta(2\ell\mid 2\ell(\ell-1))$\\
&$\Cs$&$\Delta(\ell\mid\ell(\ell-1))$\\
&$\Cns$&$\varnothing$\\
&$\Ns$&$\varnothing$\\
&$\Nns$&$\Delta(\ell\mid\ell(\ell-1))$\\
&$\Cr$&$\Delta(2\mid 2(\ell-1))$\\
\bottomrule
\end{tabular}
\caption{The set of indices of subgroups $G$ of $G_E(\ell)$ such that $G$ belongs to $M(\ell)$, for each possible non-exceptional image $G_E(\ell)=\rho_{E,\ell}(\Gal(\Qbar/\Q))$ and for the standard subgroups $M=Z,\Cs,\Cns,\Ns,\Nns,\Cr$  (the final row should be interpreted in the manner of (\ref{eqn:IGM})). Regarding notation, see (\ref{eqn:delta}, \ref{eqn:piodd}).}
\label{tab:setml}
\end{table}

\subsection{Full image}
\label{subsec:f}

\begin{lemma}
\label{lem:fz}
The set of indices of subgroups $G$ of $\GL_2(\ell)$ such that $G$ belongs to $Z(\ell)$ is the set of positive integers divisible by $\ell(\ell+1)(\ell-1)$ and dividing $\ell(\ell+1)(\ell-1)^2$.
\end{lemma}
\begin{proof}
Since finite abelian groups have subgroups of every order dividing the order of the group, and $\#Z(\ell)=\ell-1$, it follows that
\begin{equation*}
\I(\GL_2(\ell);Z(\ell))=\{[\GL_2(\ell):G]:G\le Z(\ell)\}=\Delta(\ell(\ell+1)(\ell-1)^2/\#Z(\ell)\mid\ell(\ell+1)(\ell-1)^2),
\end{equation*}
as desired.
\end{proof}

\begin{lemma}
\label{lem:fcs}
The set of indices of subgroups $G$ of $\GL_2(\ell)$ such that $G$ belongs to $\Cs(\ell)$ is the set of positive integers divisible by $\ell(\ell+1)$ and dividing an element of
\begin{equation*}
\{\ell(\ell+1)(\ell-1)^2/q:q\in\Pi(\ell-1)\}.
\end{equation*}
\end{lemma}
\begin{proof}
Since $\Cs(\ell)\simeq\F_\ell^\times\times\F_\ell^\times$, it contains non-scalar subgroups of every order dividing $\#\Cs(\ell)=(\ell-1)^2$ except for $1$, which corresponds to the trivial subgroup. The div-minimal such orders are therefore the prime divisors of $\ell-1$, and the proof concludes as in Lemma~\ref{lem:fz}.
\end{proof}

\begin{lemma}
\label{lem:fcns}
The set of indices of subgroups $G$ of $\GL_2(\ell)$ such that $G$ belongs to $\Cns(\ell)$ is the set of positive integers divisible by $\ell(\ell-1)$ and dividing an element of
\begin{equation*}
\{\ell(\ell+1)(\ell-1)^2/q:q\in\Pi_\odd(\ell+1)\cup\{2^{\nu_2(\ell-1)+1}\}\}.
\end{equation*}
\end{lemma}
\begin{proof}
Since $\#\Cns(\ell)=\ell^2-1$, we again have
\begin{equation*}
\{[\GL_2(\ell):G]:G\le\Cns(\ell)\}=\Delta(\ell(\ell-1)\mid\ell(\ell+1)(\ell-1)^2),
\end{equation*}
however, certain subgroups of $\Cns(\ell)$ may instead lie in the standard subgroup $Z(\ell)$. Any div-minimal non-scalar subgroup $G$ of $\Cns(\ell)$ with $2\ndiv\#G$ has trivial intersection with $Z(\ell)$, hence is isomorphic to its image in $\PGL_2(\ell)$, which has order dividing $\#\Cns(\ell)/(\ell-1)=\ell+1$. Since $G$ must be non-trivial, we have $\#G\in\Pi_\odd(\ell+1)$, and all such possibilities for $\#G$ are realized as $\Cns(\ell)$ is cyclic.

Otherwise, if $\#G$ is even, then $G$ must contain an element $\gamma=\mathcal{D}_\ns(x,y)\in\Cns(\ell)\setminus Z(\ell)$ squaring to an element of $Z(\ell)$; the only such elements are those for which $x=0$, which square to an element of $\epsilon Z(\ell)^2$, i.e., the non-square elements of $Z(\ell)$. Picking a primitive root $\alpha\in\F_\ell^\times$, we see that $G$ is div-minimal when $\gamma^2=\mathcal{D}_\ns(\alpha^{(\ell-1)/2^{\nu_2(\ell-1)}},0)$ and $G=\angles{\gamma}$, that is, $\#G=2^{\nu_2(\ell-1)+1}$.

Thus, the div-maximal indices in $\GL_2(\ell)$ of a non-diagonalizable subgroup of $\Cns(\ell)$ are
\begin{equation*}
\{\ell(\ell+1)(\ell-1)^2/q:q\in\Pi_\odd(\ell+1)\cup\{2^{\nu_2(\ell-1)+1}\}\},
\end{equation*}
and $\Cns(\ell)$ contains subgroups of all intermediate indices as it is cyclic.
\end{proof}

\begin{lemma}
\label{lem:nssubgpordallmult}
Let $G:=(\Z/n\Z\times\Z/n\Z)\rtimes_\varphi\Z/2\Z$, where
\begin{equation*}
\varphi\colon\Z/2\Z\to\Aut(\Z/n\Z\times\Z/n\Z),\qquad\tau\mapsto\begin{cases}
\alpha\mapsto\beta,\\
\beta\mapsto\alpha,
\end{cases}
\end{equation*}
for generators $\tau$, $\alpha$, and $\beta$ of $\Z/2\Z$ and the first and second copies of $\Z/n\Z$, respectively. Then for any $H\le G$ containing $(1,\tau)$, we have
\begin{equation*}
\{\#K:H\le K\le G\}=\Delta(\#H\mid\#G).
\end{equation*}
\end{lemma}
\begin{proof}
Let $H$ be a subgroup of $G$ containing $(0,\tau)$. Then $(h,\tau)\in H$ if and only if $(h,0)\in H$ as we may perform the group operation with $(0,\pm\tau)$ on the right. It follows that $H\simeq\pi_1(H)\rtimes_\varphi\angles{\tau}$, where $\pi_1$ denotes projection onto the first coordinate of the semi-direct product and $\pi_1(H)\le\Z/n\Z\times\Z/n\Z$. For any $d\mid [G:H]$, we would like to show that there exists some subgroup $H\le K\le G$ with $[H:K]=d$. Iterating, we may assume that $p=d$ is prime. By Goursat's lemma \cite[Thm.~4]{anderson}, $\pi_1(H)$ is given by the fiber product
\begin{equation*}
\Z/m_1\Z\times_{\Z/r\Z}\Z/m_2\Z=\{(x,y)\in\Z/m_1\Z\times\Z/m_2\Z:\psi_1(x)=\psi_2(y)\}
\end{equation*}
for some $r\mid m_i\mid n$, where the maps $\psi_i\colon\Z/m_i\Z\to\Z/r\Z$ are surjective for $i=1,2$. Since $\pi_1(H)$ must be $\varphi(\tau)$-invariant, it follows that $m_1=m=m_2$ as $\varphi(\tau)$ maps each cyclic factor of $\pi_1(H)$ bijectively to the other. This condition is also sufficient: we have $(j\cdot\frac{n}{m}\alpha,k\cdot\frac{n}{m}\beta)\in\pi_1(H)$ if and only if $j\equiv k\bmod{r}$, which is true if and only if $\varphi(\tau)(j\cdot\frac{n}{m}\alpha,k\cdot\frac{n}{m}\beta)=(k\cdot\frac{n}{m}\alpha,j\cdot\frac{n}{m}\beta)\in\pi_1(H)$. Thus,
\begin{equation*}
\#\pi_1(H)=\#\Z/r\Z\cdot\#\ker(\psi_1)\cdot\#\ker(\psi_2)=r\cdot\frac{m}{r}\cdot\frac{m}{r}=\frac{m^2}{r},
\end{equation*}
and so $p$ divides
\begin{equation*}
[G:H]=\frac{\#G}{\#H}=\frac{2n^2}{2m^2/r}=r\left(\frac{n}{m}\right)^2.
\end{equation*}
It follows that either $p\mid r$ or $p\mid \frac{n}{m}$.

In the former case, let $\kappa\colon\Z/r\Z\to\Z/(r/p)\Z$ be the unique quotient map, and let $K_1:=\Z/m\Z\times_{\Z/(r/p)\Z}\Z/m\Z$ with maps $\kappa\circ\psi_i\colon\Z/m\Z\to\Z/(r/p)\Z$. As before, $K_1$ is $\varphi(\tau)$-invariant, and $\#K_1=\frac{m^2}{r/p}=p\cdot\#\pi_1(H)$. Moreover, $\pi_1(H)\le K_1$, since $\psi_1(x)=\psi_2(y)$ implies $(\kappa\circ\psi_1)(x)=(\kappa\circ\psi_2)(y)$ for $(x,y)\in\pi_1(H)$. It follows that $K:=K_1\rtimes_\varphi\angles{\tau}$ is a subgroup of $G$ containing $H$ as an index-$p$ subgroup, as desired.

For the latter case, let $K_1:=\Z/pm\Z\times_{\Z/pr\Z}\Z/pm\Z$ (clearly $p\mid \frac{n}{r}$ as well) with maps $\psi'_i\colon\Z/pm\Z\to\Z/pr\Z$ for $i=1,2$. Then $K_1$ is again $\varphi(\tau)$-invariant, and $\#K_1=\frac{(pm)^2}{pr}=p\cdot\#\pi_1(H)$. Moreover, $\pi_1(H)\le K_1$, as the following diagram commutes for $i=1,2$, where we have let $\gamma:=\psi_i(\frac{n}{m}\alpha)$ and $\gamma':=\psi'_i(\frac{n}{pm}\alpha)$ be generators of $\Z/r\Z$ and $\Z/pr\Z$, respectively:
\begin{equation*}
\begin{tikzcd}
\Z/m\Z\arrow[hookrightarrow]{r}\arrow[twoheadrightarrow]{d}{\psi_i}&\Z/pm\Z\arrow[twoheadrightarrow]{d}{\psi'_i}\\
\Z/r\Z\arrow[hookrightarrow]{r}&\Z/pr\Z,
\end{tikzcd}
\qquad
\begin{tikzcd}
\frac{n}{m}\alpha\arrow[mapsto]{r}\arrow[mapsto]{d}&p\cdot\frac{n}{pm}\alpha\arrow[mapsto]{d}\\
\gamma\arrow[mapsto]{r}&p\gamma'.
\end{tikzcd}
\end{equation*}
Thus, $K:=K_1\rtimes_\varphi\angles{\tau}$ is again a subgroup of $G$ containing $H$ as an index-$p$-subgroup, and this completes the proof.
\end{proof}

\begin{lemma}
\label{lem:fns}
The set of indices of subgroups $G$ of $\GL_2(\ell)$ such that $G$ belongs to $\Ns(\ell)$ is the set of positive integers divisible by $\ell(\ell+1)/2$ and dividing an element of
\begin{equation*}
\{\ell(\ell+1)(\ell-1)^2/2q:q\in\Pi_\odd(\ell-1)\cup\begin{lrdcases}\{4\}&\text{ if }4\mid(\ell-1),\\
\{2\}&\text{ otherwise.}\end{lrdcases}\}.
\end{equation*}
\end{lemma}
\begin{proof}
Since $\#\Ns(\ell)=2(\ell-1)^2$, we have
\begin{equation*}
\{[\GL_2(\ell):G]:G\le\Ns(\ell)\}\subseteq\Delta(\ell(\ell+1)/2\mid\ell(\ell+1)(\ell-1)^2),
\end{equation*}
however, certain subgroups of $\Ns(\ell)$ may instead lie in the standard subgroup $\Cs(\ell)$. Thus, any subgroup $G$ belonging to $\Ns(\ell)$ must be non-abelian, so it contains an anti-diagonal element; since any anti-diagonal element of $\Ns(\ell)$ has even order, $\#G$ is even. If $\#G=4$, then $G$ is abelian, hence conjugate to a subgroup of $\Cs(\ell)$ or $\Cns(\ell)$; thus, if the former is not true, then $G$ must be cyclic. If $\ell\equiv 1\bmod 4$, then any element of $\Ns(\ell)$ of order $4$ is diagonalizable by Lemma~\ref{lem:fcns}, so such a $G$ must be conjugate to a subgroup of $\Cs(\ell)$. In this case,
\begin{equation}
\label{eqn:d8}
\angles{\mathcal{A}_\s(1,1),\mathcal{D}_\s(1,-1)}\simeq D_8
\end{equation}
is a non-abelian subgroup of order $8$, hence div-minimal. Otherwise, if $\ell\equiv 3\bmod 4$, pick a primitive root $\alpha\in\F_\ell^\times$; the subgroup $\angles{\mathcal{A}_\s(\alpha,\alpha^{(\ell-3)/2})}$ is cyclic of order $4$, and cannot be conjugate to a subgroup of $\Cs(\ell)$ as $-1\in\F_\ell^\times$ is not a square, so it is div-minimal.

Now, for each $q\in\Pi_\odd(\ell-1)$, the subgroup
\begin{equation*}
\angles{\mathcal{A}_\s(1,1),\mathcal{D}_\s(\alpha^{(\ell-1)/q},\alpha^{-(\ell-1)/q})}\simeq D_{2q}
\end{equation*}
is non-abelian and of order $2q$, hence div-minimal. Thus,
\begin{equation}
\label{eqn:divminima}
2\cdot(\Pi_\odd(\ell-1)\cup\{4\text{ if }4\mid(\ell-1),\text{ else }2\})
\end{equation}
is the set of div-minima of $\I(\GL_2(\ell);\Ns(\ell))$, for which the corresponding div-minimal subgroups each contain $\mathcal{A}_\s(1,1)$ (if $\ell\equiv 3\bmod 4$, then this is not the case for the div-minimal subgroup of order $4$; however, any other integer in $\Delta(4\mid 2(\ell-1)^2)$ will either be divisible by some other element of (\ref{eqn:divminima}), or divisible by $8$. Since the order-$8$ subgroup (\ref{eqn:d8}) contains $\mathcal{A}_\s(1,1)$, the rest of the argument goes through unchanged). Since
\begin{equation*}
\Ns\simeq(\Z/(\ell-1)\Z\times\Z/(\ell-1)\Z)\rtimes\Z/2\Z
\end{equation*}
with the non-trivial element $\mathcal{A}_\s(1,1)\in\Z/2\Z$ corresponding to the automorphism in Lemma~\ref{lem:nssubgpordallmult}, it follows that
\begin{equation*}
\Delta(\{\ell(\ell+1)/2\},\{\ell(\ell+1)(\ell-1)^2/2q:q\in\Pi_\odd(\ell-1)\cup\{4\text{ if }4\mid(\ell-1),\text{ else }2\}\})
\end{equation*}
is in fact the set of possible indices of subgroups of $\GL_2(\ell)$ belonging to $\Ns(\ell)$.
\end{proof}

\begin{lemma}
\label{lem:nnssubgpordallmult}
Let $G:=\Z/n\Z\rtimes_\varphi A$, where $A$ is abelian. Then for any $H\le G$ containing $A$, we have
\begin{equation*}
\{\#K:H\le K\le G\}=\Delta(\#H\mid\#G).
\end{equation*}
The same is true for any $H\le G$ projecting trivially onto $A$. 
\end{lemma}
\begin{proof}
As in Lemma~\ref{lem:nssubgpordallmult}, the projection $\pi_1(H)$ is a subgroup of $\Z/n\Z$, hence for any $d\mid [G:H]=n/\#\pi_1(H)$, there exists a unique subgroup $K_1\le\Z/n\Z$ of order $d\cdot\#\pi_1(H)$ containing $\pi_1(H)$. Letting $\alpha$ be a generator of $\Z/n\Z$, observe that any automorphism of $\Z/n\Z$ is of the form $\alpha\mapsto\alpha^i$, where $(i,n)=1$, hence any cyclic subgroup of $\Z/n\Z$ is $\varphi(a)$-invariant for each $a\in A$. Thus, $K:=K_1\rtimes_\varphi A$ is a subgroup of $G$ containing $H$ as an index-$d$ subgroup, as desired.

For the second assertion, note that any $d\mid [G:H]$ factors as a product of $d_1\mid [\Z/n\Z:\pi_1(H)]$ and $d_2\mid\# A$. Letting $K_1$ be the unique order $d_1\cdot\#\pi_1(H)$ subgroup of $\Z/n\Z$ containing $\pi_1(H)$ and $K_2$ be a subgroup of $A$ of order $d_2$, we see that $K:=K_1\rtimes_\varphi K_2$ is again a subgroup of $G$ containing $H$ as an index-$d$ subgroup.
\end{proof}

\begin{lemma}
\label{lem:index2subgpint}
Let $H\le G$ such that $H\not\le N$, for some index-$2$ subgroup $N$ of $G$. Then $[H:H\cap N]=2$.
\end{lemma}
\begin{proof}
Let $\gamma\in H\cap(G\setminus N)$. Then $G=N\sqcup\gamma N$, and left multiplication by $\gamma$ is a bijective map of sets on $G$, mapping $N$ to $\gamma N$ and vice versa. Thus, $\gamma\cdot(H\cap N)\subset H\cap\gamma N$ has cardinality $\#(H\cap N)$, hence $[H:H\cap N]\ge 2$. Moreover, $\gamma\cdot(H\cap\gamma N)\subset H\cap N$ has cardinality $\#(H\cap\gamma N)$, hence $[H:H\cap N]\le 2$, and the result follows.
\end{proof}

\begin{lemma}
\label{lem:fnns}
The set of indices of subgroups $G$ of $\GL_2(\ell)$ such that $G$ belongs to $\Nns(\ell)$ is the set of positive integers divisible by $\ell(\ell-1)/2$ and dividing $\ell(\ell+1)(\ell-1)^2/2$.
\end{lemma}
\begin{proof}
Since $\#\Nns(\ell)=2(\ell^2-1)$, we have
\begin{equation*}
\{[\GL_2(\ell):G]:G\le\Nns(\ell)\}\subseteq\Delta(\ell(\ell-1)/2\mid\ell(\ell+1)(\ell-1)^2),
\end{equation*}
however, certain subgroups of $\Nns(\ell)$ may instead lie in the standard subgroup $\Cns(\ell)$. By Lemma~\ref{lem:index2subgpint}, any subgroup of $\Nns(\ell)$ that is not contained in $\Cns(\ell)$ must have order divisible by $2$. One such order-$2$ subgroup is $\angles{\mathcal{A}_\ns(1,0)}$ (which is not conjugate to $\angles{\mathcal{D}_\ns(1,0)}$, the unique order-$2$ subgroup of $\Cns(\ell)$), so by Lemma~\ref{lem:nnssubgpordallmult}, the set of possible indices of subgroups of $\GL_2(\ell)$ belonging to $\Nns(\ell)$ is
\begin{equation*}
\Delta([\GL_2(\ell):\Nns(\ell)]\mid\#\GL_2(\ell)/2),
\end{equation*}
as desired.
\end{proof}

\subsection{Normalizer of a split Cartan}
\label{subsec:s}

\begin{lemma}
\label{lem:sz}
The set of indices of subgroups $G$ of $\Ns(\ell)$ such that $G$ belongs to $Z(\ell)$ is the set of positive integers divisible by $2(\ell-1)$ and dividing $2(\ell-1)^2$.
\end{lemma}
\begin{proof}
Since finite abelian groups have subgroups of every order dividing the order of the groups, it follows that
\begin{equation*}
\I(\Ns(\ell);Z(\ell))=\{[\Ns(\ell):G]:G\le Z(\ell)\}=\Delta([\Ns(\ell):Z(\ell)]\mid\#\Ns(\ell)),
\end{equation*}
as desired.
\end{proof}

\begin{lemma}
\label{lem:scs}
The set of indices of subgroups $G$ of $\Ns(\ell)$ such that $G$ belongs to $\Cs(\ell)$ is the set of positive integers divisible by $2$ and dividing an element of
\begin{equation*}
\{2(\ell-1)^2/q:q\in\Pi(\ell-1)\}.
\end{equation*}
\end{lemma}
\begin{proof}
The result follows as in Lemma~\ref{lem:fcs} after noting that $[\Ns(\ell):\Cs(\ell)]=2$.
\end{proof}

\begin{lemma}
\label{lem:scns}
The set of indices of subgroups $G$ of $\Ns(\ell)$ such that $G$ belongs to $\Cns(\ell)$ is the set of positive integers divisible by $\ell-1$ and dividing $(\ell-1)^2/2^{\nu_2(\ell-1)}$.
\end{lemma}
\begin{proof}
Let $\alpha\in\F_\ell^\times$ be a primitive root. Clearly, $\angles{\mathcal{A}_\s(1,\alpha)}\le\Ns(\ell)$ has order $2(\ell-1)$, hence it is an index-$(\ell-1)$ subgroup belonging to $\Cns(\ell)$ (as the maximal cyclic subgroups of $\Cs(\ell)$ are of order $\ell-1$). Since any anti-diagonal element of $\Ns(\ell)$ squares to a diagonal element, $2(\ell-1)$ is the maximal order of an element of $\Ns(\ell)$, hence any subgroup of $\Ns(\ell)$ belonging to $\Cns(\ell)$ has index divisible by $\ell-1$ in $\Ns(\ell)$. Now, since $\mathcal{A}_\s(1,\alpha)^2\in Z(\ell)$ (the same is true for any other order $2(\ell-1)$ cyclic subgroup of $\Ns(\ell)$), a subgroup of $\angles{\mathcal{A}_\s(1,\alpha)}$ is conjugate to a subgroup of $Z(\ell)$ if and only if its order divides $\ell-1$; thus, a subgroup of $\angles{\mathcal{A}_\s(1,\alpha)}$ is not conjugate to a subgroup of $Z(\ell)$ if and only if it contains the unique subgroup of order $2^{\nu_2(\ell-1)+1}$. Since $\angles{\mathcal{A}_\s(1,\alpha)}$ is cyclic, it contains a unique subgroup of each possible order, and therefore
\begin{equation*}
\Delta(\ell-1\mid 2(\ell-1)^2/2^{\nu_2(\ell-1)+1})
\end{equation*}
is indeed the set of indices of subgroups of $\Ns(\ell)$ belonging to $\Cns(\ell)$.
\end{proof}

\begin{lemma}
\label{lem:sns}
The set of indices of subgroups $G$ of $\Ns(\ell)$ such that $G$ belongs to $\Ns(\ell)$ is the set of positive integers dividing an element of
\begin{equation*}
\{(\ell-1)^2/q:q\in\Pi_\odd(\ell-1)\cup\{4\text{ if }4\mid(\ell-1),\text{ else }2\}\}.
\end{equation*}
\end{lemma}
\begin{proof}
By Lemma~\ref{lem:fns}, this is obtained by dividing all indices of $\mathcal{I}(\GL_2(\ell);\Ns(\ell))$ by
\begin{equation*}
[\GL_2(\ell):\Ns(\ell)]=\frac{\ell(\ell+1)}{2},
\end{equation*}
as
\begin{equation*}
\{G\le\GL_2(\ell):G\text{ belongs to }\Ns(\ell)\}=\{G\le\Ns(\ell):G\text{ belongs to }\Ns(\ell)\},
\end{equation*}
which clearly yields the desired set.
\end{proof}

\begin{lemma}
\label{lem:snns}
The set of indices of subgroups $G$ of $\Ns(\ell)$ such that $G$ belongs to $\Nns(\ell)$ is the set of positive integers divisible by $(\ell-1)/2$ and dividing $(\ell-1)^2$.
\end{lemma}
\begin{proof}
Observe that
\begin{equation*}
Z(\ell)\cdot\{\mathcal{D}_\s(1,1),\mathcal{D}_\s(1,-1),\mathcal{A}_\s(\epsilon,1),\mathcal{A}_\s(-\epsilon,1)\}=\Ns(\ell)\cap\Nns(\ell)\le\Ns(\ell)
\end{equation*}
is a subgroup of order $4(\ell-1)$ contained in $\Nns(\ell)$, isomorphic to $\Z/2(\ell-1)\Z\rtimes\Z/2\Z$. By Lemma~\ref{lem:index2subgpint}, any subgroup of $\Nns(\ell)$ is either contained in $\Cns(\ell)$ or has an index-$2$ subgroup contained in $\Cns(\ell)$, hence by Lemma~\ref{lem:scns}, $\#G\mid 2\cdot 2(\ell-1)=4(\ell-1)$ for any $G\le\Ns(\ell)$ belonging to $\Nns(\ell)$. Moreover, by Lemma~\ref{lem:index2subgpint}, any subgroup of $\Ns(\ell)$ belonging to $\Nns(\ell)$ must have order divisible by $2$; as in Lemma~\ref{lem:fnns}, one such order-$2$ subgroup is $\angles{\mathcal{A}_\ns(1,0)}$, so by Lemma~\ref{lem:nnssubgpordallmult}, the set of possible indices of subgroups of $\Ns(\ell)$ belonging to $\Nns(\ell)$ is
\begin{equation*}
\Delta(\#\Ns(\ell)/4(\ell-1)\mid\#\Ns(\ell)/2),
\end{equation*}
as desired.
\end{proof}

\subsection{Normalizer of a non-split Cartan}
\label{subsec:n}

\begin{lemma}
\label{lem:nz}
The set of indices of subgroups $G$ of $\Nns(\ell)$ such that $G$ belongs to $\Cs(\ell)$ is the set of positive integers divisible by $2(\ell+1)$ and dividing $2(\ell^2-1)$.
\end{lemma}
\begin{proof}
The result follows as in Lemma~\ref{lem:sz} after noting that $Z(\ell)=\{\mathcal{D}_\ns(x,0):x\in\F_\ell^\times\}\le\Nns(\ell)$.
\end{proof}

\begin{lemma}
\label{lem:ncs}
The set of indices of subgroups $G$ of $\Nns(\ell)$ such that $G$ belongs to $\Cs(\ell)$ is the set of positive integers divisible by $\ell+1$ and dividing $\ell^2-1$.
\end{lemma}
\begin{proof}
Clearly,
\begin{equation}
\label{eqn:gp}
Z(\ell)\cdot\{\mathcal{D}_\ns(1,0),\mathcal{A}_\ns(1,0)\}=\Nns(\ell)\cap\Cs(\ell)\le\Nns(\ell)
\end{equation}
is a subgroup of order $2(\ell-1)$ (isomorphic to $\Z/(\ell-1)\Z\times\Z/2\Z$) belonging to $\Cs(\ell)$. For any $G\le\Nns(\ell)$ belonging to $\Cs(\ell)$, the proof of Lemma~\ref{lem:fnns} shows that $G$ is conjugate to a subgroup of
\begin{equation*}
Z(\ell)\cdot\{\mathcal{D}_\s(1,1),\mathcal{A}_\s(1,1)\}\subseteq\Ns(\ell),
\end{equation*}
since $G$ has exponent dividing $\ell-1$. This group is further conjugate to (\ref{eqn:gp}) (via $\smallmat{1}{1}{-1}{1}$), so $\#G\mid 2(\ell-1)$. Now, a subgroup of (\ref{eqn:gp}) belongs to $\Cs(\ell)$ if and only if contains a multiple of $\mathcal{A}_\ns(1,0)$, that is, if and only if its order is divisible by $2$. Thus,
\begin{equation*}
\I(\Nns(\ell);\Cs(\ell))=\Delta(\#\Nns(\ell)/(2(\ell-1))\mid\#\Nns(\ell)/2),
\end{equation*}
as desired.
\end{proof}

\begin{lemma}
\label{lem:ncns}
The set of indices of subgroups $G$ of $\Nns(\ell)$ such that $G$ belongs to $\Cns(\ell)$ is the set of positive integers divisible by $2$ and dividing an element of
\begin{equation*}
\{2(\ell^2-1)/q:q\in\Pi_\odd(\ell+1)\cup\{2^{\nu_2(\ell-1)+1}\}\}.
\end{equation*}
\end{lemma}
\begin{proof}
The proof is as that of Lemma~\ref{lem:sns}, applied to Lemma~\ref{lem:fcns}.
\end{proof}

\begin{lemma}
\label{lem:nns}
The set of indices of subgroups $G$ of $\Nns(\ell)$ such that $G$ belongs to $\Ns(\ell)$ is the set of positive integers divisible by $(\ell+1)/2$ and dividing \((\ell^2-1)/2^{\nu_2(\ell-1)}\).
\end{lemma}
\begin{proof}
Observe that
\begin{equation*}
Z(\ell)\cdot\{\mathcal{D}_\ns(1,0),\mathcal{D}_\ns(0,1),\mathcal{A}_\ns(1,0),\mathcal{A}_\ns(0,1)\}=\Nns(\ell)\cap\Ns(\ell)\le\Nns(\ell)
\end{equation*}
is a subgroup of order $4(\ell-1)$ contained in $\Ns(\ell)$. Moreover, by Lemma~\ref{lem:index2subgpint}, any subgroup of $\Ns(\ell)$ is either contained in $\Cs(\ell)$ or has an index-$2$ subgroup contained in $\Cs(\ell)$, hence by Lemma~\ref{lem:ncs}, $\#G\mid 2\cdot 2(\ell-1)=4(\ell-1)$ for any $G\le\Nns(\ell)$ belonging to $\Ns(\ell)$. Now, any subgroup of $\Nns(\ell)$ belonging to $\Ns(\ell)$ of order $4(\ell-1)$ is conjugate to $\Nns(\ell)\cap\Ns(\ell)$, so it suffices to find the div-minimal subgroups of $\Nns(\ell)\cap\Ns(\ell)$ belonging to $\Ns(\ell)$. Any such subgroup is either cyclic and generated by a non-diagonalizable matrix or non-abelian. Thus, in the first case, it must be generated by an element of $Z(\ell)\cdot\mathcal{D}_\ns(0,1)$, or, if $\ell\equiv 1\bmod 4$, an element of $Z(\ell)\cdot\mathcal{A}_\ns(0,1)$; regardless, the div-minimal order is $2^{\nu_2(\ell-1)+1}$ as in Lemma~\ref{lem:fcns}. In the latter case, it must be generated by two elements, each from a distinct set in
\begin{equation*}
\{Z(\ell)\cdot\mathcal{D}_\ns(0,1),Z(\ell)\cdot\mathcal{A}_\ns(1,0),Z(\ell)\cdot\mathcal{A}_\ns(0,1)\}.
\end{equation*}
However, given any two such elements, multiplying them gives an element of the unchosen set, hence we may suppose without loss of generality that our subgroup is given by $\angles{\gamma,\delta}$, where $\gamma\in Z(\ell)\cdot\mathcal{D}_\ns(0,1)$ and $\delta\in Z(\ell)\cdot\mathcal{A}_\ns(1,0)$. By Lemma~\ref{lem:index2subgpint}, $\#\angles{\gamma,\delta}$ is div-minimal when $\#(\angles{\gamma,\delta}\cap\Cs(\ell))$ is, so since $2\mid\#\angles{\delta}$, we may assume $\delta=\mathcal{A}_\ns(1,0)$, and since $2^{\nu_2(\ell-1)+1}\mid\#\angles{\gamma}$ as in Lemma~\ref{lem:fcns}, we may assume $\gamma=\mathcal{D}_\ns(0,\beta)$, where $\beta^2=\epsilon^{-1}\alpha^{(\ell-1)/2^{\nu_2(\ell-1)}}$ and $\alpha\in\F_\ell^\times$ is a primitive root. Now, since $\gamma\delta\gamma^{-1}=\mathcal{A}_\ns(-1,0)\in\angles{\gamma,\delta}$, it follows that
\begin{align*}
\#(\angles{\gamma,\delta}\cap\Cs(\ell))&=\#(\{\mathcal{D}_\ns(1,0),\mathcal{D}_\ns(-1,0),\mathcal{A}_\ns(1,0),\mathcal{A}_\ns(-1,0)\}\cdot\angles{\gamma^2})\\
&=\frac{\#\angles{\mathcal{A}_\ns(1,0),\mathcal{A}_\ns(-1,0)}\cdot\#\angles{\gamma^2}}{\#(\angles{\mathcal{A}_\ns(1,0),\mathcal{A}_\ns(-1,0)}\cap\angles{\gamma^2})}\\
&=\frac{4\cdot 2^{\nu_2(\ell-1)}}{2}\\
&=2^{\nu_2(\ell-1)+1},
\end{align*}
so Lemma~\ref{lem:index2subgpint} implies that
\begin{equation*}
\#\angles{\gamma,\delta}=2\cdot\#(\angles{\gamma,\delta}\cap\Cs(\ell))=2\cdot 2^{\nu_2(\ell-1)+1}=2^{\nu_2(\ell-1)+2},
\end{equation*}
which is divisible by $2^{\nu_2(\ell-1)+1}$, the div-minimal order in the first case. Thus, by Lemma~\ref{lem:nnssubgpordallmult}, we have
\begin{equation*}
\I(\Nns(\ell);\Ns(\ell))=\Delta(\#\Nns(\ell)/(4(\ell-1))\mid\#\Nns(\ell)/2^{\nu_2(\ell-1)+1}),
\end{equation*}
as desired.
\end{proof}

\begin{lemma}
\label{lem:nnns}
The set of indices of subgroups $G$ of $\Nns(\ell)$ such that $G$ belongs to $\Nns(\ell)$ is the set of positive integers dividing $\ell^2-1$.
\end{lemma}
\begin{proof}
The proof is as that of Lemma~\ref{lem:sns}, applied to Lemma~\ref{lem:fnns}.
\end{proof}

\subsection{Borel cases}
\label{subsec:b}

\begin{lemma}
\label{lem:borelsubgp}
Let $G\le B(\ell)$, and suppose that $G$ does not contain an element of order $\ell$. Then $G$ is conjugate to a subgroup of $\Cs(\ell)$.
\end{lemma}
\begin{proof}
Let $H$ denote the image of $G$ in $\PGL_2(\ell)$, and observe that $S:=\{x:\smallmat{x}{y}{0}{1}\in H\}$ forms a subgroup of $\F_\ell^\times$, of which we let $a$ be a generator (as it must be cyclic). Then there exists some $\gamma:=\smallmat{a}{b}{0}{1}\in H$, and for any $n$, we have
\begin{equation*}
\gamma^n=\begin{pmatrix}a^n&b(1+a+\cdots+a^{n-1})\\0&1\end{pmatrix},
\end{equation*}
so $a\ne 1$ unless $b=0$, as otherwise $\ord(\gamma)=\ell$ (and all its lifts to $B(\ell)$ have order divisible by $\ell$). If $a\ne 1$, then $1+a+\cdots+a^{n-1}=0$ if and only if
\begin{equation*}
(1-a)(1+a+\cdots+a^{n-1})=1-a^n=0,
\end{equation*}
that is, $\ord(a)\mid n$, from which it follows that $\ord(\gamma)=\ord(a)\mid (\ell-1)$. Letting $\gamma':=\smallmat{a'}{b'}{0}{1}\in H$ and $r$ be such that $a^{-r}=a'$ (which exists as $a'\in S$), we have
\begin{equation*}
\gamma^r\gamma'=\begin{pmatrix}a^ra'&a^rb'+b(1+a+\cdots+a^{r-1})\\0&1\end{pmatrix},
\end{equation*}
and so $b'=-ba^{-r}(1+a+\cdots+a^{n-1})$ since otherwise $\gamma^r\gamma'$ has order $\ell$ as above. But then $\gamma'=\gamma^{-r}$, implying that $H=\angles{\gamma}$ is cyclic. Since every element of $B(\ell)$ has an eigenvalue in $\F_\ell$, whereas no element of $\Cns(\ell)\setminus Z(\ell)$ has any eigenvalues over $\F_\ell$ (as those of $\mathcal{D}_\ns(x,y)\in\Cns(\ell)$ are $x\pm y\sqrt{\epsilon}\in\overline{\F_\ell}$), $G$ is conjugate to a subgroup of $\Cs(\ell)$ by Proposition~\ref{prop:subs}, as desired.
\end{proof}

\begin{lemma}
\label{lem:bm}
The set of indices of subgroups $G$ of $G(\ell)$, $H_1(\ell)$, or $H_2(\ell)$ such that $G$ belongs to a standard subgroup $M(\ell)$ is the set of positive integers divisible by $2\ell$ (resp.\ $\ell$) and dividing $2\ell(\ell-1)$ (resp.\ $\ell(\ell-1)$) if $M$ is $Z$ (resp.\ $\Cs$); the empty set if $M$ is $\Cns$ or $\Ns$; and the set of positive integers divisible by $\ell$ and dividing $\ell(\ell-1)$ if $M$ is $\Nns$.
\end{lemma}
\begin{proof}
Evidently, $[G(\ell):G(\ell)\cap\Cs(\ell)]=\ell$, and since $G(\ell)$ contains an element of order $\ell$, any subgroup of $G(\ell)$ conjugate to a subgroup of $\Cs(\ell)$ must have index divisible by $\ell$ in $G(\ell)$. Since
\begin{equation*}
G(\ell)\cap\Cs(\ell)=Z(\ell)\cdot\{\mathcal{D}_s(1,1),\mathcal{D}_s(1,-1)\},
\end{equation*}
the proof of Lemma~\ref{lem:ncs} shows that 
\begin{equation*}
\I(G(\ell);\Cs(\ell))=\Delta(\ell\mid\#G(\ell)/2)=\Delta(\ell\mid\ell(\ell-1)).
\end{equation*}
Also, $\I(G(\ell);Z(\ell))=\Delta(2\ell\mid 2\ell(\ell-1))$ as before. Performing the analogous computations for $H_1(\ell)$ and $H_2(\ell)$ and taking unions gives the result if $M$ is $Z$ or $\Cs$.

For the case where $M$ is $\Cns$, observe that any subgroup of $\Cns(\ell)$ conjugate to a subgroup of $\Cs(\ell)$ is contained in $Z(\ell)$: an element $\mathcal{D}_\ns(x,y)\in\Cns(\ell)$ is diagonalizable if and only if
\begin{equation*}
(2x)^2-4(x^2-\epsilon y^2)=4\epsilon y^2\in(\F_\ell)^2,
\end{equation*}
that is, $y=0$. Thus, $\I(G(\ell),H_1(\ell),H_2(\ell);\Cns(\ell))=\varnothing$ by Lemma~\ref{lem:borelsubgp}.

The case where $M$ is $\Ns$ is immediate from Lemma~\ref{lem:borelsubgp}.

For the case where $M$ is $\Nns$, we must identify the non-scalar subgroups of $G(\ell)$ that are conjugate to a subgroup of $\Nns(\ell)$. These cannot contain an element of order $\ell$, so we may assume that they are subgroups of
\begin{equation*}
G(\ell)\cap\Cs(\ell)=Z(\ell)\cdot\{\mathcal{D}_\ns(1,0),\mathcal{A}_\ns(1,0)\}\subseteq\Nns(\ell).
\end{equation*}
It immediately follows that $\I(G(\ell);\Nns(\ell)=\I(G(\ell);\Cs(\ell))$; performing a similar analysis for $H_1(\ell)$ and $H_2(\ell)$ and taking unions then gives the desired result.
\end{proof}

\section{Proof of Theorem~\ref{thm:3}}

Let $E/K$ be an elliptic curve, for some number field $K$. The Mordell--Weil group $E(K)$ contains a point of order $\ell$ if and only if $\rho_{E,\ell}(\Gal(\Kbar/K))$ fixes a nonzero element of $\F_\ell^2$. Thus, as in the proof of Theorem~\ref{thm:1}, there exists a degree-$d$ number field $K$ and an elliptic curve $E/\Q$ for which $E(K)$ contains a point of order $\ell$ if and only if $\rho_{E,\ell}(\Gal(\Qbar/\Q))$ contains a degree-$t$ subgroup fixing a nonzero element of $\F_\ell^2$ for some $t\mid d$. The same is true of an elliptic curve $E/K$ with $j(E)\in\Q\setminus\{0,1728\}$ if and only if $\rho_{E,\ell}(\Gal(\Kbar/K))$ is a twist of a degree-$t$ subgroup of $\rho_{E',\ell}(\Gal(\Qbar/\Q))$ for some elliptic curve $E'/\Q$ and $t\mid d$, and it fixes a nonzero element of $\F_\ell^2$. For a subgroup $G\le\GL_2(\ell)$, define
\begin{equation}
\label{eqn:iij}
\begin{split}
\I(G_1,\ldots,G_n)&:=\{[G_i:H]:H\le G_i\text{ such that }H'v=v\text{ for some }0\ne v\in\F_\ell^2,\\
&\qquad\text{twist }H'\text{ of }H\text{, and }1\le i\le n\},\\
\I_\Q(G_1,\ldots,G_n)&:=\{[G_i:H]:H\le G_i\text{ such that }Hv=v\text{ for some }0\ne v\in\F_\ell^2\text{ and }1\le i\le n\},\\
\end{split}
\end{equation}
and let $\divmin(S)$ denote the div-minima of a set $S$. By the above, it remains to determine the sets
\begin{equation*}
\mathcal{T}(\ell)=\divmin\bigcup_{E/\Q}\I(G_E(\ell)),\qquad\mathcal{T}_\Q(\ell)=\divmin\bigcup_{E/\Q}\I_\Q(G_E(\ell)),
\end{equation*}
where the elliptic curves in each union are assumed to satisfy $j(E)\ne 0,1728$. We begin by determining the sets $\divmin\I(G_E(\ell))$ individually for every possible image $G_E(\ell)=\rho_{E,\ell}(\Gal(\Qbar/\Q))$; the sets $\divmin\I_\Q(G_E(\ell))$ may be determined analogously.

If $E$ does not have CM, then $G_E(\ell)$ is either $\GL_2(\ell)$ or one of the $63$ exceptional groups listed in Table~\ref{tab:exim} by our assumption of Conjecture~\ref{conj:sutherland}. The results for each exceptional subgroup were computed \cite{code,sutherland2016} using Magma, and the div-minima of the unions of the two sets (\ref{eqn:iij}) of indices over each \(\Excep(\ell)\) are the sets $\I(\Excep(\ell))$ and $\I_\Q(\Excep(\ell))$, respectively, listed in Table~\ref{tab:eejl}. The set $\divmin\I(\GL_2(\ell))$ is determined by Lemma~\ref{lem:fj} below, and appears in Table~\ref{tab:setl}.

Now suppose $E$ has CM. As in the proof of Theorems~\ref{thm:1} and \ref{thm:2}, $G_E(\ell)$ is conjugate to $\Ns(\ell)$ (resp.\ $\Nns(\ell)$) if $\ell\in\mathcal{P}$ (resp.\ $\ell\in\mathcal{Q}$); to either $\Ns(\ell)$, $\Nns(\ell)$, $G(\ell)$, $H_1(\ell)$, or $H_2(\ell)$ if $\ell\in\{3,7,11,19,43,67,163\}$, with each case occurring for some such $E$; and otherwise to either $\Ns(\ell)$ or $\Nns(\ell)$, with each case again occurring for some such $E$ \cite[Prop.~1.14]{zywina2015}. The set $\divmin\I(\Ns(\ell))$ is determined by Lemma~\ref{lem:sj}; the set $\divmin(\I(\Nns(\ell))$ by Lemma~\ref{lem:nj}; and the set $\divmin\I(G(\ell),H_1(\ell),H_2(\ell))$ by Lemma~\ref{lem:bj}. All appear in Table~\ref{tab:setl}.

Finally, the sets appearing in Tables~\ref{tab:tl} and \ref{tab:tjl} are determined by letting $\mathcal{T}(\ell)$ be the div-minima of the set
\begin{equation*}
\I(\Excep(\ell),\GL_2(\ell))\cup\begin{lrdcases}\I(\Ns(\ell))&\text{if }\ell\in\mathcal{P},\\
\I(\Nns(\ell))&\text{if }\ell\in\mathcal{Q},\\
\I(\Ns(\ell),\Nns(\ell),G(\ell),H_1(\ell),H_2(\ell))&\text{if }\ell\in\{3,7,11,19,43,67,163\},\\
\I(\Ns(\ell),\Nns(\ell))&\text{otherwise}.
\end{lrdcases},
\end{equation*}
and similarly for the sets $\mathcal{T}^\CM(\ell)$ and $\mathcal{T}^\nonCM(\ell)$.

We now turn to the task of proving the lemmas used above, which will complete the proof of Theorem~\ref{thm:3}.

\begin{table}[htb!]
$$
\begin{tabu}{rll}
\toprule
\ell&\divmin\I(\Excep(\ell))&\divmin\I_\Q(\Excep(\ell))\\\midrule
2,3,5,7&1&1\\
11&5&5\\
13&2,3&3,4\\
17&4&8\\
37&6&12\\
\text{other}&\varnothing&\varnothing\\
\bottomrule
\end{tabu}
$$
\caption{div-Minima of the set of indices of subgroups $G$ of some exceptional image $G_E(\ell)=\rho_{E,\ell}(\Gal(\Qbar/\Q))$ listed in Table~\ref{tab:exim} such that a twist of $G$ (or, in the third column, $G$ itself) fixes a nonzero vector of $\F_\ell^2$.}
\label{tab:eejl}
\end{table}

\begin{table}
$$
\begin{tabu}{rll}
\toprule
G_E&\divmin\I(G_E(\ell))&\divmin\I_\Q(G_E(\ell))\\\midrule
\GL_2&(\ell^2-1)/2&\ell^2-1\\
\Ns&\ell-1&2(\ell-1)\\
\Nns&(\ell^2-1)/2&\ell^2-1\\
G,H_1,H_2&(\ell-1)/2&(\ell-1)/2\\
\bottomrule
\end{tabu}
$$
\caption{div-Minima of the set of indices of subgroups $G$ of $G_E(\ell)$ such that a twist of $G$ (or, in the third column, $G$ itself) fixes a nonzero vector of $\F_\ell^2$, for each non-exceptional possible image $G_E(\ell)=\rho_{E,\ell}(\Gal(\Qbar/\Q))$ (the final row should be interpreted in the manner of (\ref{eqn:iij})).}
\label{tab:setl}
\end{table}

\begin{lemma}
\label{lem:fj}
The set of indices of subgroups $G$ of $\GL_2(\ell)$ such that a twist of $G$ fixes a nonzero element of $\F_\ell^2$ has div-minima $\{(\ell^2-1)/2\}$.
\end{lemma}
\begin{proof}
Any subgroup of $\GL_2(\ell)$ fixing each element of a $1$-dimensional $\F_\ell$-subspace of $\F_\ell^2$ must be conjugate to a subgroup fixing $\angles{\smallvect{1}{0}}$. Such a subgroup must be upper triangular, with a $1$ in the upper left-hand entry of each matrix, hence it is a subgroup of
\begin{equation*}
T:=\{\smallmat{1}{a}{0}{b}:a\in\F_\ell,b\in\F_\ell^\times\}\le B(\ell)\le\GL_2(\ell),
\end{equation*}
which has order $\ell(\ell-1)$, hence index $\ell^2-1$ in $\GL_2(\ell)$.

Now, since $-1$ does not fix any nontrivial element of $\F_\ell^2$, we need only consider index-$2$ subgroups $H\le\angles{G,-1}$ with $-1\notin H$, where $H\le T$ and $G\le\GL_2(\ell)$, by the previous paragraph. Both $[\GL_2(\ell):H]$ and $[\GL_2(\ell):G]$ are either $[\GL_2(\ell):\angles{G,-1}]$ or $2\cdot[\GL_2(\ell):\angles{G,-1}]$, and evidently $[\GL_2(\ell):G]$ is div-minimal when $-1\in G$ and $[\GL_2(\ell):H]=2\cdot[\GL_2(\ell):G]$. Taking $G=\angles{T,-1}$ then gives a subgroup of $\GL_2(\ell)$ of index $(\ell^2-1)/2$ with $[G:T]=2$. For any other such $G$ and $H$, we have
\begin{equation*}
(\ell^2-1)\mid[\GL_2(\ell):T]\mid[\GL_2(\ell):H]\mid 2\cdot[\GL_2(\ell):G],
\end{equation*}
implying $(\ell^2-1)/2\mid [\GL_2(\ell):G]$, as desired.
\end{proof}

\begin{lemma}
\label{lem:sj}
The set of indices of subgroups $G$ of $\Ns(\ell)$ such that a twist of $G$ fixes a nonzero element of $\F_\ell^2$ has div-minima $\{(\ell-1)\}$.
\end{lemma}
\begin{proof}
The proof is identical to that of Lemma~\ref{lem:fj}, with $\GL_2(\ell)$ replaced by $\Ns(\ell)$ and $T$ replaced by each of the subgroups identified in \cite[Lem.~6.6]{lozanorobledo2013} in turn.
\end{proof}

\begin{lemma}
\label{lem:nj}
The set of indices of subgroups $G$ of $\Nns(\ell)$ such that a twist of $G$ fixes a nonzero element of $\F_\ell^2$ has div-minima $\{(\ell^2-1)/2\}$.
\end{lemma}
\begin{proof}
The proof is identical to that of Lemma~\ref{lem:fj}, with $\GL_2(\ell)$ replaced by $\Nns(\ell)$ and $T$ replaced by each of the subgroups identified in \cite[Lem.~7.4]{lozanorobledo2013} in turn.
\end{proof}

\begin{lemma}
\label{lem:bj}
Suppose $\ell\equiv 3\bmod 4$. The set of indices of subgroups $G$ of $G(\ell)$, $H_1(\ell)$, or $H_2(\ell)$ such that a twist of $G$ fixes a nonzero element of $\F_\ell^2$ has div-minima $\{(\ell-1)/2\}$.
\end{lemma}
\begin{proof}
The characteristic polynomial of an element $\smallmat{a}{b}{0}{a}\in G(\ell)$ is $(\lambda-a)^2$, hence it has eigenvalue $1$ if and only if $a=1$. Similarly, the characteristic polynomial of an element $\smallmat{a}{b}{0}{-a}\in G(\ell)$ is $\lambda^2-a^2$, hence it has eigenvalue $1$ if and only if $a=\pm 1$. If $a=-1$, then it fixes every element of the space $\angles{\smallvect{b-1}{1}}$, hence it is the only nontrivial element of $G(\ell)$ fixing every element of this space, and $\angles{\smallmat{-1}{b}{0}{1}}\le G(\ell)$ is a subgroup of order $2$. Otherwise, if $a=1$, then both types of elements fix every element of $\angles{\smallvect{1}{0}}$, and together they form the index-$(\ell-1)$ subgroup
\begin{equation*}
\{\smallmat{1}{b}{0}{\pm 1}:b\in\F_\ell\}\le G(\ell),
\end{equation*}
which is maximal as its order, $2\ell$, is divisible by $2$.

The same reasoning shows that the subgroup $\{\smallmat{1}{b}{0}{\pm 1}:b\in\F_\ell\}\le H_1(\ell)$ is maximal, and it also has order $2\ell$, hence index $(\ell-1)/2$.

For $H_2(\ell)$, we obtain the same characteristic polynomials, but in this case elements $\smallmat{-a}{b}{0}{a}\in G(\ell)$ with $a=1$ each generate order-$2$ subgroups fixing $\angles{\smallvect{b-1}{1}}$, and the elements of $H_2(\ell)$ with upper left-hand entry~$1$ form a subgroup of order $\ell$, since $\ell\equiv 3\bmod{4}$ by assumption and therefore $-1\notin(\F_\ell^\times)^2$. Thus, the two div-minimal indices of such subgroups of $H_2(\ell)$ are $\ell-1$ and $\ell(\ell-1)/2$.

Repeating the proof of Lemma~\ref{lem:fj} with $\GL_2(\ell)$ replaced by $G(\ell)$ and $T$ replaced by each of the subgroups of $G(\ell)$ identified above in turn shows that $\I(G(\ell))=\{(\ell-1)/2\}$. Since $-1\notin H_1(\ell),H_2(\ell)$ as $\ell\equiv 3\bmod{4}$ by assumption, any twist of a subgroup $H$ of $H_1(\ell)$ or $H_2(\ell)$ is a subgroup of $H$, hence $\I(H_1(\ell))=\I_\Q(H_1(\ell))=\{(\ell-1)/2\}$ and $\I(H_2(\ell))=\I_\Q(H_2(\ell))=\{\ell-1,\ell(\ell-1)/2\}$. Thus, among $G(\ell)$, $H_1(\ell)$, and $H_2(\ell)$, the div-minimal index is $(\ell-1)/2$.
\end{proof}

\bibliographystyle{merlin}
\bibliography{references}

\end{document}